%% file: ms.tex
\makeatletter
\def\cl@chapter{}
\makeatother

\documentclass[]{svjour3}                     
\smartqed  
\usepackage[utf8]{inputenc}
\usepackage{mathptmx}      
\usepackage{imakeidx}
\usepackage{hyperref}
\usepackage{graphicx}
\usepackage{xspace}
\usepackage{xparse}
\usepackage{enumitem}
\usepackage{booktabs}
\usepackage{csquotes}

\usepackage{amsmath}
\usepackage{amssymb}
\usepackage{commath}
\usepackage{mathtools}
\usepackage{tikz-cd}
\usepackage{dsfont}

\usepackage[capitalize,nameinlink]{cleveref}

\DeclareGraphicsExtensions{.pdf}

\input{definitions.tex}

\crefname{theorem}{Theorem}{Theorems}
\crefname{hypothesis}{Hypothesis}{Hypotheses}
\crefname{notation}{Notation}{Notations}
\crefname{definition}{Definition}{Definitions}
\crefname{remark}{Remark}{Remarks}
\crefname{example}{Example}{Examples}
\crefname{lemma}{Lemma}{Lemmata}
\crefname{proposition}{Proposition}{Propositions}
\crefname{assumption}{Assumption}{Assumptions}

\crefformat{equation}{(#2#1#3)}
\crefrangeformat{equation}{(#3#1#4) to~(#5#2#6)}
\crefmultiformat{equation}{(#2#1#3)}{ and~(#2#1#3)}{, (#2#1#3)}{ and~(#2#1#3)}

\numberwithin{equation}{section}

\makeindex

\begin{document}

\title{Coupled Systems of Linear Differential-Algebraic and Kinetic Equations
with Application to the Mathematical Modelling of Muscle Tissue.}

\titlerunning{Linear Differential-Algebraic and Kinetic Equations
	with Applications to Muscle Tissue Models}        

\author{Steffen Plunder \and Bernd Simeon
}

\institute{
    Steffen Plunder \at
    University of Vienna,
    Faculty of Mathematics, \\
    Oskar-Morgensternplatz 4,
    1090 Wien \\
    \email{steffen.plunder@univie.ac.at} 
    \and
    Bernd Simeon \at
    TU Kaiserslautern, 
    Felix Klein Zentrum für Mathematik, \\
    Paul-Ehrlich Straße,
    67663 Kaiserslautern \\
    Tel.: +49 631 205 5310 \\
    \email{simeon@mathematik.uni-kl.de}
}

\date{Received: date / Accepted: date}

\maketitle

\begin{abstract}
We consider a coupled system composed of a linear differential-algebraic equation (DAE) and a linear large-scale system of ordinary differential equations where the latter stands for the dynamics of numerous identical particles. 
Replacing the discrete particles by a kinetic equation for a particle density, we obtain in the mean-field limit the new class of partially kinetic systems.

We investigate the influence of constraints on the kinetic theory of those systems and present necessary adjustments. 
We adapt the mean-field limit to the DAE model and show that index reduction and the mean-field limit commute.
As a main result, we prove Dobrushin's stability estimate for linear systems. The estimate implies convergence of the mean-field limit and provides a rigorous link between the particle dynamics and their kinetic description.

Our research is inspired by mathematical models for muscle tissue where the macroscopic behaviour is governed by the equations of continuum mechanics, often discretised by the finite element method, and the microscopic muscle contraction process is described by Huxley’s sliding filament theory. The latter represents a kinetic equation that characterises the state of the actin-myosin bindings in the muscle filaments. 
Linear partially kinetic systems are a simplified version of such models, with focus on the constraints.

\end{abstract}

\setcounter{tocdepth}{2}
\section{Introduction}
\label{sec:intro}

Differential-algebraic equations (DAEs)\index{DAE}\index{Differential equation!algebraic|see {DAE}} and kinetic equations \index{Kinetic theory|seealso {Mean-field limit}} are usually considered as separate and quite independent topics. While DAEs stem from models that are in some sense constrained, kinetic theory deals with identical particles such as atoms or molecules and their mutual interaction. In this work, we introduce a problem class that combines these two mathematical structures.
More precisely, we demonstrate how to couple a macroscopic component with a kinetic system using algebraic constraints.

For given microscopic laws, the kinetic description of a particle system is obtained by the mean-field limit which replaces the discrete particle states by a particle distribution.
Application of the mean-field limit \index{Mean-field limit} comes with a loss of information: Individual particle positions are lost and only their statistical distribution is available.
This gives rise to a challenge in the coupling process since we cannot impose an algebraic constraint on the individual particle positions of a kinetic system.
However, if we know the microscopic laws governing the particles of a kinetic system,
we can impose the algebraic-constraint on the microscopic level for each particle and then apply the mean-field limit to obtain new kinetic equations for the particles.
We refer to the resulting system as the partially kinetic system.\index{Partially kinetic systems}

To the best of our knowledge, there is no kinetic theory for systems where DAEs describe the microscopic law. In order to provide a rigorous theory for partially kinetic systems, we extend ideas from classical kinetic theory \cite{spohn1980kinetic,spohn2012large,jabinReviewMeanField2014,golseDynamicsLargeParticle2016}.
To streamline the presentation, we restrict ourself to linear systems. 
\color{black}

Important examples for partially kinetic system are mathematical models for muscle contraction. Muscle tissue, with all its supporting tissue (macroscopic component), contracts due to the accumulated force of numerous actin-myosin cross-bridges\index{Cross-bridges} (particles). In this specific case, the kinetic theory of cross-bridges without the coupling is already well-studied and led to the famous Huxley model \index{Huxley model|seealso {Sliding filament theory}} \cite{huxleyMuscleStructureTheories1957,keenerMathematicalPhysiology2009,zahalakDistributionmomentApproximationKinetic1981}.
On the other hand, models from continuum mechanics are today in use to simulate the muscle contraction at the macro-scale in combination with the finite element method. For the coupling of both scales, simplifications and ad-hoc procedures are used so far \cite{bolMicromechanicalModellingSkeletal2008,BrCP96,heidlaufMultiscaleContinuumModel2016,heidlaufContinuummechanicalSkeletalMuscle2017} that call for  a theoretical foundation. 

This article is organized as follows:
\cref{sec:dae_model} presents a strongly simplified DAE model for muscle cells with attached actin-myosin cross-bridges and derives an equivalent ODE formulation for the DAE model.
Next, \cref{sec:muscles_kin} derives formally partially kinetic equations for the DAE model.
Basics of kinetic theory are outlined during the application of the mean-field limit onto the ODE formulation.
For the DAE formulation, the mean-field limit requires modification.
To justify the formal computations rigorously, \cref{sec:mean_field_limit} adapts and proves Dobrushin's stability estimate for linear partially kinetic systems.
With regard to the application fields, \cref{sec:generalisations} sketches possible generalisations of linear partially kinetic systems, while
\cref{sec:numerical_examples} provides details about the numerical implementation of the simulations presented in this article. The numerical challenge of partially kinetic systems is demonstrated by an example in which energy conservation is violated by the discretisation.

\section{A Differential-Algebraic Model for Muscle Cells with Attached Cross-Bridges}
\label{sec:dae_model}

The emergence of macroscopic effects from microscopic properties is a central theme in kinetic theory. \index{Kinetic theory} In laymen terms, emergence describes how the big picture arises from the laws that govern the model at a smaller scale. Understanding this transition is essential in many biological applications \cite{resatKineticModelingBiological2009}.
Muscle tissue consists of millions of small contractible molecules called actin-myosin cross-bridges. Kinetic theory allows the up-scaling 
of these microscopic units to the organ level and provides a means to
derive macroscopic models for muscle tissue. Most macroscopic models focus on the emergence of a contraction force as the result of the synchronization between muscle cells
\cite{herzogSkeletalMuscleMechanics2000,keenerMathematicalPhysiology2009,maDistributionmomentModelEnergetics1991}.
However, there are applications where more than just the macroscopic contraction force is of interest.

One example is vibrational medicine, in particular, the 
medical therapy concept called \emph{Matrix-Rhythm-Therapy}\index{Matrix-Rhythm-Therapy}
\cite{Randoll97} that
treats diseased muscle tissue by vibrational stimulation in a specific frequency range. 
In order to understand this therapy approach, it is crucial to study how the mechanical stimulation influences the physiological health of cells. In laymen terms: How does the big picture influence the small scale?
A first mathematical model for the interplay between mechanics and the physiology of muscle cells was proposed in \cite{simeonModelMacroscaleDeformation2009}. 

We extend this work in the direction of more detailed physiological models for muscle cells that are based on the sliding filament theory for cross-bridges.
In mathematical terms, this requires an understanding of muscles at both, the micro and the macro scale. To study the influence of mechanics on the physiology of muscle cells, the coupling between mechanical properties and physiological models is essential.
In the following, we will study a prototype of a system which couples a physiological model for cross-bridges with a prototypical mechanical system. 
We have to mention that Ma and Zahalak \cite{zahalakNonaxialMuscleStress1996,maDistributionmomentModelEnergetics1991} already studied cross-bridge dynamics with kinetic and thermodynamic methods and also extended their cross-bridge models for refined coupling with continuum models. In contrast, we study a simpler model and focus on mathematical details of the coupling. In \cref{subsec:compare_estab_models}, we relate our mathematical analysis to established models.

\subsection{Sliding Filament Theory for Cross-Bridges}
\label{subsec:sliding_filament}

Compared to many other biological phenomena, the contraction of muscles cells is a relatively well-studied field \cite{herzogSkeletalMuscleMechanics2000,howardMechanicsMotorProteins2001,keenerMathematicalPhysiology2009}.
For a mathematical introduction to muscle models, we refer to  \cite{howardMechanicsMotorProteins2001,keenerMathematicalPhysiology2009}. 
\emph{Sliding filament theory}\index{Sliding filament theory} \cite{herzogSkeletalMuscleMechanics2000,huxleyMuscleStructureTheories1957} is the mainstream theory to explain muscle contraction.

In its simplest form, sliding filament theory \index{Sliding filament theory} suggests that muscle cells consist of parallel myosin and actin filaments, as visualised in \cref{fig:myosin_actin}.
On each actin filament, small binding sides allow myosin heads to attach and to form a bridge between both filaments, a so-called \emph{cross-bridge}.\index{Cross-bridges}
Due to the molecular configuration of newly attached cross-bridges, they pull the two filaments such that they slide alongside each other, which causes a shortening of the muscle cell. This pulling step is called a power stroke. After each power stroke, the myosin head can unbind from the binding side, release the ADP (adenosine diphosphate) molecule and obtain new free energy by hydrolyzing another 
ATP (adenosine triphosphate) molecule. 

The cycling of binding, power stroke, unbinding and resetting of myosin heads is called the cross-bridge cycle. \index{Cross-bridges!cycling}
Since numerous muscle cells contract due to this mechanism, the whole muscle tissue contracts on the macroscopic scale.
The rate at which the cross-bridge cycling takes place controls the contraction strength.
The contraction process varies for different muscle types. However, blocking and unblocking the binding sides at the actin filaments is always part of the control mechanism.
In skeletal muscle tissue, the periodic release of calcium ions unblocks the binding sides. A higher frequency of calcium ion bursts leads to a stronger contraction.

\begin{figure}[h]
	\centering
	\includegraphics[width=0.8\textwidth]{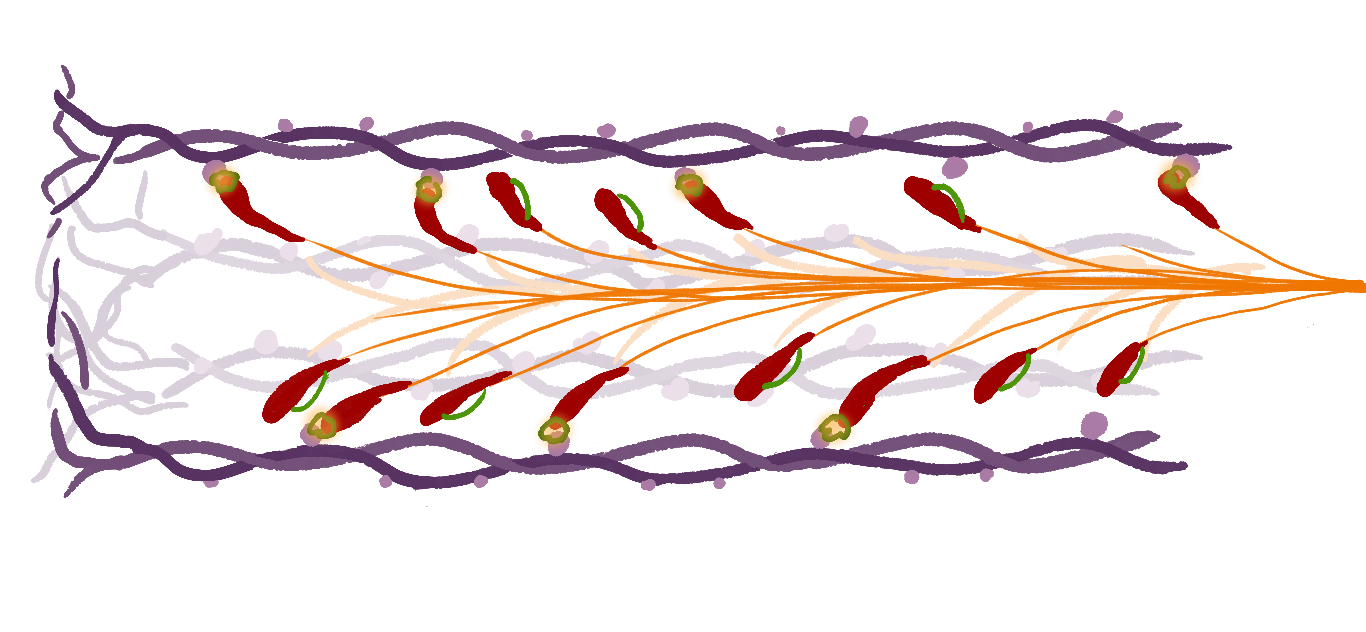}
	\caption{Sketch of the parallel actin filaments (purple, outside) and myosin filaments (orange, central). The myosin heads (red) can attach to binding sides at the actin filament, which forms a so-called cross-bridge. Skeletal muscle fibers are a large array of parallel actin-myosin filaments.}
	\label{fig:myosin_actin}
\end{figure}

From the variety of available mathematical models, we extract the common core, which is given by the sliding filament theory with cross-bridges modelled as springs \cite{herzogSkeletalMuscleMechanics2000,huxleyMuscleStructureTheories1957,keenerMathematicalPhysiology2009,zahalakDistributionmomentApproximationKinetic1981}.
To obtain a linear model, we assume that the springs are linear, which is valid for some models \cite{huxleyMuscleStructureTheories1957,zahalakDistributionmomentApproximationKinetic1981}, but not the case for more detailed models \cite{huxleyProposedMechanismForce1971}, \cite[Section 15.6]{keenerMathematicalPhysiology2009}.
We also simplify the model radically by considering only the attached cross-bridges. Hence, the actual cross-bridge cycling does not take place in the system we present. However, \cref{subsec:compare_estab_models} outlines possible extensions, which we neglect for most of the exposition to avoid distraction from the main mathematical ideas.

\begin{figure}[h]
	\centering
	\includegraphics[width=\textwidth]{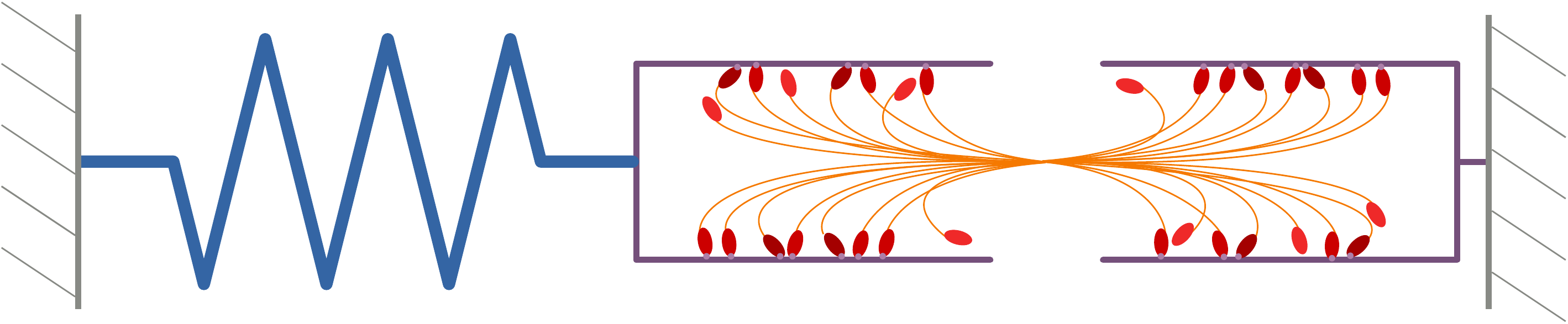}
	\caption{Model for the coupling between a macroscopic linear spring (blue) and microscopic myosin filaments (orange) with their corresponding pair of actin filaments (purple).}
	\label{fig:actin_myosin_spring}
\end{figure}

\subsection{A Differential-Algebraic Model for Attached Cross-Bridges}
\label{subsec:dae_model}

Without further ado, we present the mathematical model for attached cross-bridges\index{Cross-bridges!attached}
in the presence of constraints.
Our goal is to model a muscle cell which is coupled to a macroscopic linear spring, as displayed in \cref{fig:actin_myosin_spring}.
It is sufficient to model only one half of the actin-myosin filaments from \cref{fig:actin_myosin_spring}, since their arrangement is mirror-symmetrical.

We define the dimensions
$n_r \coloneqq 1, n_q \coloneqq 1$, where $n_r$ denotes the dimension of the macroscopic spring and $n_q$ denotes the degrees of freedom of a single cross-bridge. However, throughout this article, we will continue to distinguish between $\RR$, $\Rr$ and $\Rq$ to indicate real numbers and position variables in the according spaces.  The reader is welcome to read this section with $\nq, \nr \geq 1$ in mind. While the coupled cross-bridge model leads to the one-dimensional case, we simultaneously include an abstract model where the macroscopic system has $\nr$ degrees of freedom and each particle has $n_q$ degrees of freedom.

To model the macroscopic spring, we use $r \in \Rr$ as the extension of a linear spring with mass $\Mr$ and force $\Fr(r) = - \kappar r$.

For the microscopic model, we label the attached cross-bridges with $j=1,\dots,\N$, where $\N$ is the total number of cross-bridges. 
The extension of a single cross-bridge is denoted by $\qd_j \in \Rq$ and each cross-bridge is a linear spring with mass $\Mq$ and force $\Fq(\qd_j) = -\kappaq \qd_j$. Because a single cross-bridge is very light compared to the macroscopic spring, the dynamics of cross-bridges are typically fast compared to the macroscopic spring. The small individual mass is compensated by the large number of cross-bridges $\N$, which is a crucial parameter of the system. The situation is sketched in \cref{fig:simplified_actin_myosin_spring}. 

\begin{figure}[h]
	\centering
	\includegraphics[width=\textwidth]{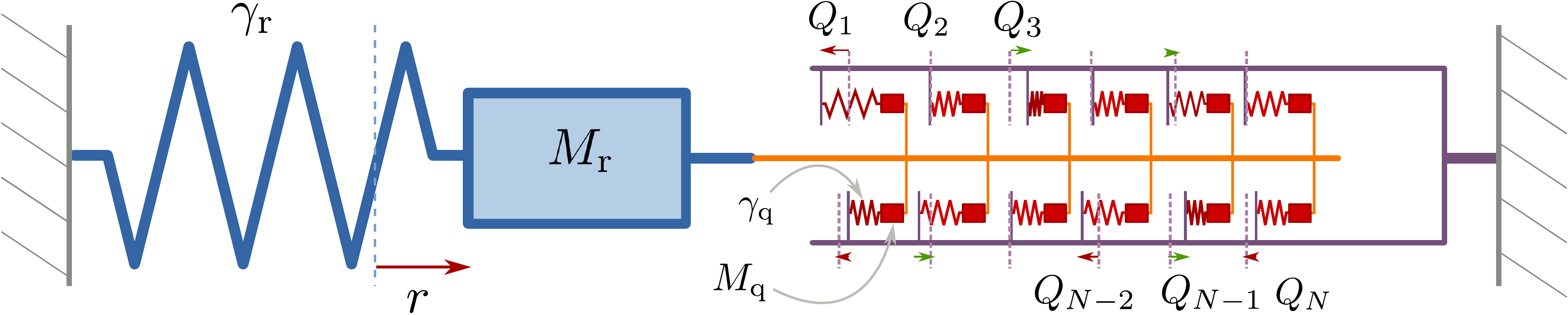}
	\caption{Simplified model for attached cross-bridges (red) between parallelly sliding actin and myosin filaments coupled with a macroscopic linear spring (blue).}
	\label{fig:simplified_actin_myosin_spring}
\end{figure}

For the abstract model $\nr,\nq \geq 1$, we define
the function
\[
g(r,\qd_j) \coloneqq \qd_j + \cDphi  r,
\]
where $\cDphi \in \RR^{\nq \times \nr}$  is an arbitrary, possibly singular matrix.
In the cross-bridge model, the macroscopic spring and the actin filament are considered to be fixed to the walls at both sides, as displayed in \cref{fig:simplified_actin_myosin_spring}. Therefore, we require the total length to remain constant and pick $\cDphi \coloneqq -1$ in the one-dimensional case $\nr = \nq = 1$.
For each cross-bridge, we define the linear constraint \index{Algebraic constraints!linear} as
\[
g(r,\qd_j) = g(\ri,\qdi_j)  \quad \text{for~} j = 1,\dots,N,
\] 
where $\ri \in \Rr$ and $\qdi_j \in \Rq$ denote the initial states of the macroscopic system and the cross-bridges.
The corresponding Lagrangian multipliers \index{Lagrangian multipliers} are denoted by $\lambda_1,\dots, \lambda_N \in \Rq$. 
Overall, we arrive at the following linear differential-algebraic system that models a linear spring coupled to a sliding actin-myosin filament pair with $N$ cross-bridges
\begin{align}
\Mr \ddot r &= - \kappar r - \sum_{i=1}^{\N} \cDphi^T \lambda_i , 
\label{eq:dm_newton_r}
\\
\Mq \ddot \qd_j &= - \kappaq \qd_j - \lambda_j \quad \text{for~} j=1,\dots,\N, 
\label{eq:dm_newton_qj}
\\
\qd_j + \cDphi r  &= \qdi_j + \cDphi \ri  \quad \text{for~} j=1,\dots,\N 
\label{eq:dm_constr_ind3}
\end{align}
with initial conditions
\[
r(0) = \ri \in \Rr, \quad \dot r(0) = \si \in \Rr \quad \text{and} \quad \qd_j(0) = \qdi_j \in \Rq \quad \text{for~} j=1,\dots,\N.
\]
In the following, we will refer to \cref{eq:dm_newton_r,eq:dm_newton_qj,eq:dm_constr_ind3} as the \emph{DAE formulation}.
There is no initial condition for the velocities $\dot \qd_j$, since the constraint implies the compatibility condition
\[
\dot \qd_j(0) = -\cDphi \si  \quad \text{for~} j=1,\dots,\N.
\]

We require the mass matrices $\Mr \in \RR^{\nr \times \nr}, \Mq \in \RR^{\nq \times \nq}$ to be positive definite. We remark that $\pd{g}{\qd_j}(r,\qd_j) = \mathds{1} \in \RR^{\nq \times \nq}$, which implies the full-rank condition required for local existence and uniqueness of solutions \cite[Section VII.I, Eq. (1.16)]{hairerStiffDifferentialalgebraicProblems2010}.

\subsection{Derivation of an Effective Balance Law via Index Reduction and Elimination of Multipliers}
\label{subsec:dm_index_reduction}

The system \cref{eq:dm_newton_r,eq:dm_newton_qj,eq:dm_constr_ind3} has differential index 3 \cite[Section VII.I]{hairerStiffDifferentialalgebraicProblems2010}\cite{BrCP96}.
Due to the particular structure, it is possible to eliminate the Lagrange multipliers and derive an ODE formulation.
Differentiating \cref{eq:dm_constr_ind3} with respect to time yields 
\begin{align}
\dot \qd_j = -\cDphi \dot r
\label{eq:dm_constr_ind2_first}
\intertext{and}
\ddot \qd_j = -\cDphi \ddot r.
\label{eq:dm_constr_ind1}
\end{align}
Using \cref{eq:dm_constr_ind1}, we solve \cref{eq:dm_newton_qj} for $\lambda_j$ and insert the result into \cref{eq:dm_newton_r}, which leads to
\begin{align*}
\Mr \ddot r 
&= -\kappar r - \sum_{i=1}^\N \cDphi^T \left( -\kappaq \qd_i - \Mq \ddot \qd_i \right) \\
&= -\kappar r - \sum_{i=1}^\N \cDphi^T \left( -\kappaq \qd_i + \Mq \cDphi \ddot r \right).
\end{align*}
After collecting the acceleration terms on the left-hand side, one obtains
\begin{align}
\underbrace{ \left( \Mr + \sum_{i=1}^\N \cDphi^T\Mq \cDphi \right) }_{\eqqcolon \Meffn} \ddot r = \underbrace{-\kappar r + \sum_{i=1}^\N \cDphi^T \kappaq \qd_i}_{\eqqcolon \Feffn(r,\qd_1,\dots,\qd_N)}. 
\label{eq:dm_newton_eff}
\end{align}
This system of ordinary differential equations describes the \emph{effective} balance of forces after elimination of the constraint equation, and thus we use the subscript ${}_{\text{eff}}$.

In \cref{eq:dm_newton_eff}, the Lagrangian multipliers are eliminated, but the equation is not closed since $\qd_i$ is needed to compute $\Feffn$.
We employ \cref{eq:dm_constr_ind2_first} to generate a first order differential equation for all $\qd_j$, i.e.
\begin{align}
\dot \qd_j = - \cDphi \dot r \quad \text{for~} j=1,\dots,\N.
\label{eq:dm_constr_ind2}
\end{align}
Now, the equations \cref{eq:dm_newton_eff,eq:dm_constr_ind2}
form a closed linear ordinary differential equation (ODE) which we will call the \emph{ODE formulation}. There are other ODE formulations, but we prefer \cref{eq:dm_newton_eff,eq:dm_constr_ind2} since this form leads to a direct derivation of the mean-field PDE in \cref{subsec:muscle_mf_pde}. 

A numerical simulation of \cref{eq:dm_newton_eff,eq:dm_constr_ind2} is presented in \cref{fig:linear_disc}. For the simulation, the initial conditions of the cross-bridge extensions $\qd_i$ are samples of a normal distribution. 
The cross-bridges influence the effective mass and the effective force of the macroscopic system. This influence leads to a shift of the macroscopic system's equilibrium to $r_0 \approx 1.5$ instead of $r_0 = 0$.
Since the constraint is $r - \qd_j = \ri - \qdi_j$, the trajectories of the cross-bridge extensions just differ by constant shifts.
For details on the numerical method and the used parameters, we refer to \cref{sec:numerical_examples}.

\begin{figure}[h] \centering
	\includegraphics[width=0.48\textwidth]{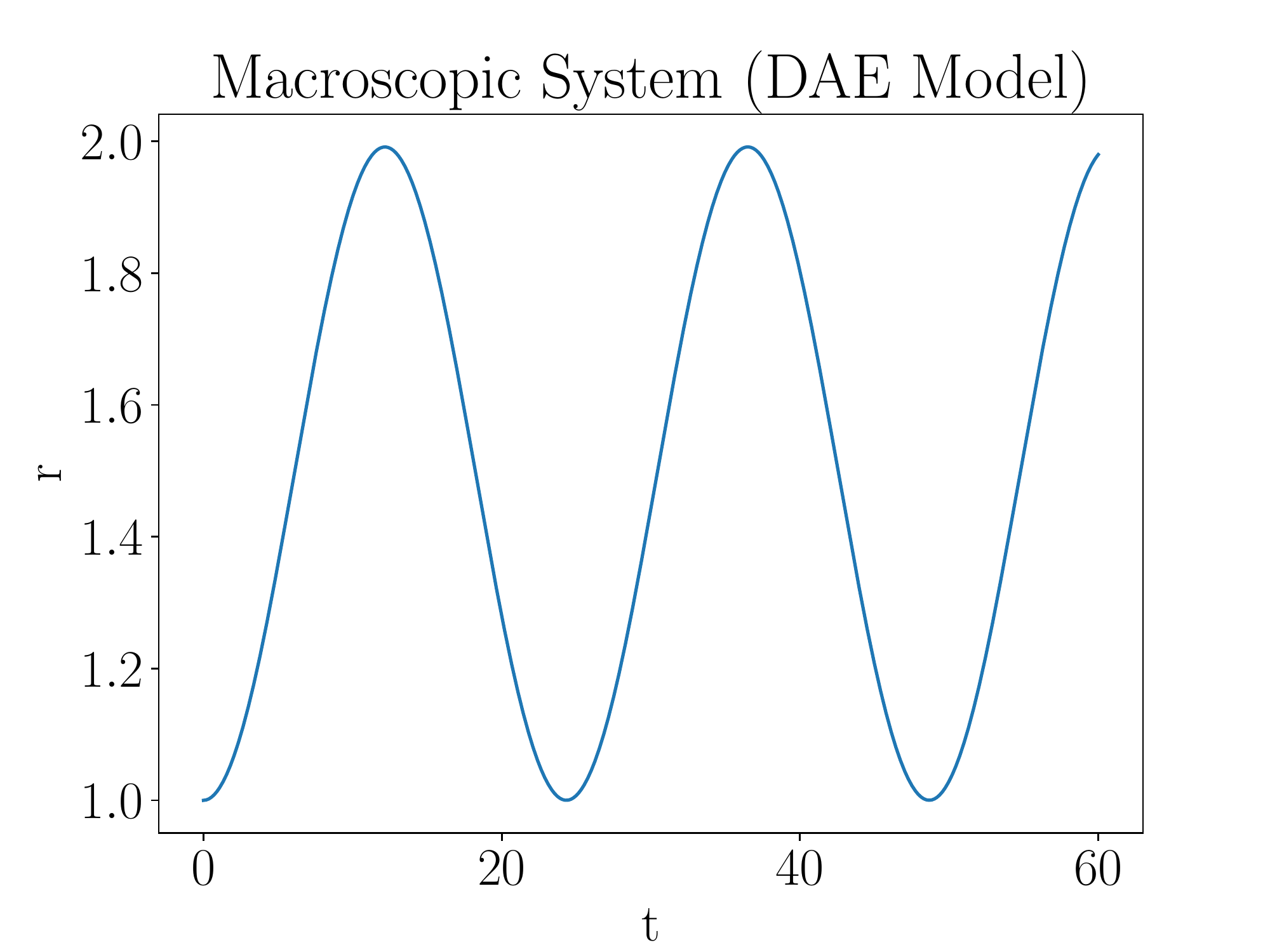}
	\hfill
	\includegraphics[width=0.48\textwidth]{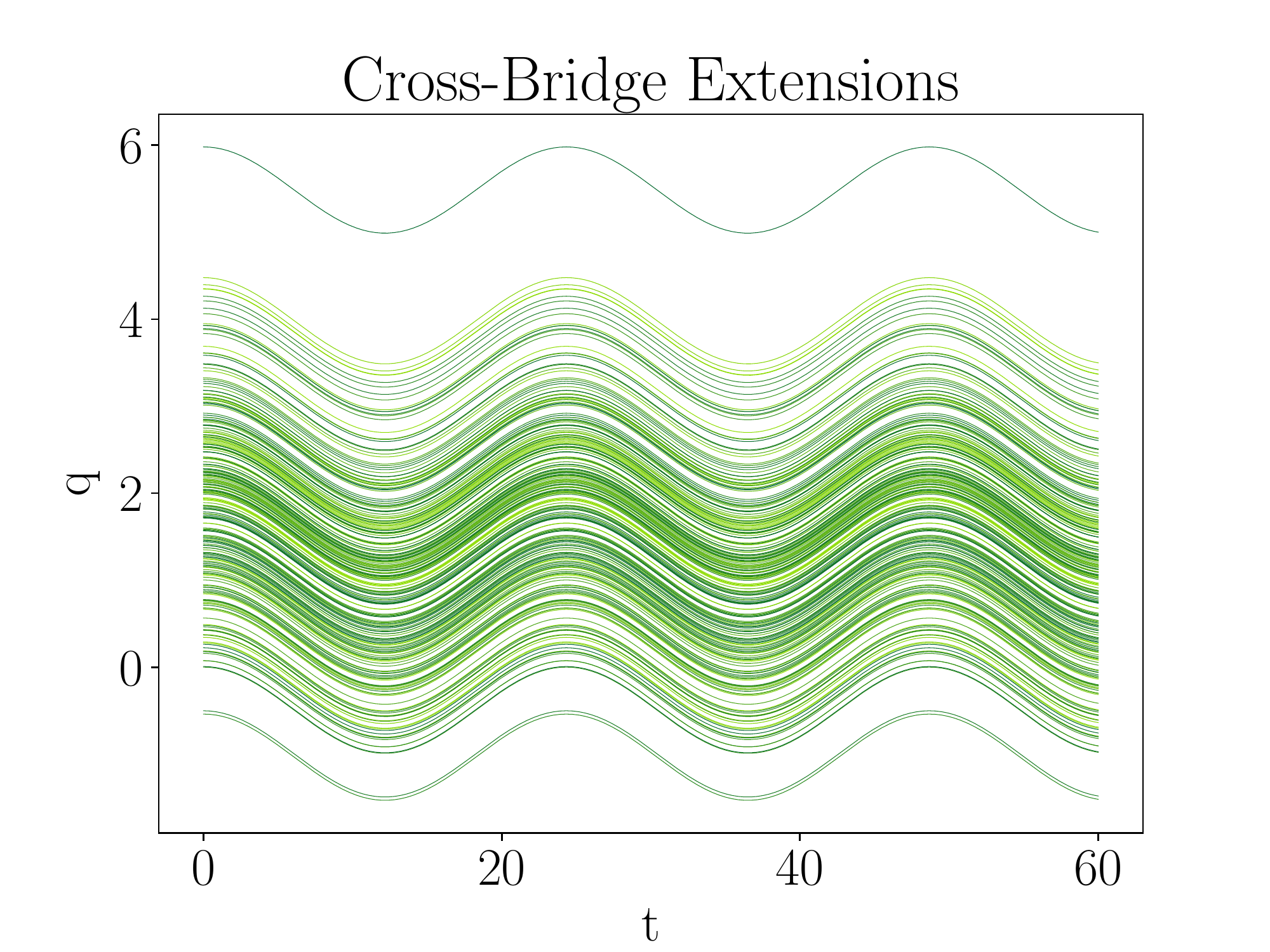}
	\caption{The trajectory of the macroscopic system (left) and the cross-bridge extensions (right).}
	\label{fig:linear_disc}
\end{figure}

\subsection{Explicit Solutions}
The linear constraint \cref{eq:dm_constr_ind3} 
can be solved for $\qd_i$, which yields 
\begin{align}
\qd_i = -\cDphi r + \cDphi \ri + \qdi_i. \notag
\end{align}
With this formula, we can reformulate \cref{eq:dm_newton_eff} as
\begin{align}
\Meffn \ddot r = -\gamma_\mathrm{eff}^{(\N)} ( r - r_\mathrm{0}^{(\N)})
\label{eq:dm_explicit_sol}
\end{align}
with the effective stiffness
\[
\gamma_\mathrm{eff}^{(\N)} \coloneqq \kappar + \N \cDphi^T \kappaq
\]
and the new equilibrium
\[
r_0^{(\N)} = \cDphi^T \kappaq \sum_{i=1}^N (\cDphi \ri + \qdi_i).
\]
Equation \cref{eq:dm_explicit_sol} has well-known explicit solutions.
The system \cref{eq:dm_newton_r,eq:dm_newton_qj,eq:dm_constr_ind3} 
is therefore a benchmark case for partially kinetic systems.

In the setting of more realistic muscle models, explicit solutions are not available any more, since attachment and detachment of cross-bridges lead to a switched system where the number of cross-bridges $\N$ changes over time, as outlined in \cref{subsec:compare_estab_models}.
Moreover, if the constraint function $g(r,\qd_j)$ is nonlinear with respect to $\qd_j$, then explicit solutions are not known in general.
These extensions are discussed in \cref{sec:generalisations}. 

Despite the existence of explicit solutions, we will continue 
without using \cref{eq:dm_constr_ind3} to solve for $\qd_j$, since this benchmark-setting is convenient to demonstrate the influence of constraints on kinetic theory. The calculations in \cref{sec:muscles_kin,sec:mean_field_limit} generalise well to relevant extensions as sketched in \cref{sec:generalisations}. Only in proofs of this article, we will use \cref{eq:dm_constr_ind3} explicitly.

\section{Partially Kinetic Model for Muscle Cells with Attached Cross-Bridges}
\label{sec:muscles_kin}

The kinetic theory for cross-bridges was investigated already in the early eighties \cite{huxleyMuscleStructureTheories1957}.
The first approaches suggested to model attached cross-bridges as linear springs, 
while many refinements have been introduced later on and are still today subject of
current research \cite{herzogSkeletalMuscleMechanics2017}.
To compute the contraction force, all models known to us
 \cite{bolMicromechanicalModellingSkeletal2008,heidlaufMultiscaleContinuumModel2016,heidlaufContinuummechanicalSkeletalMuscle2017,herzogSkeletalMuscleMechanics2000,keenerMathematicalPhysiology2009} assume implicitly
that the kinetic equations remain valid without modification in the presence of constraints, cf. \cite{herzogSkeletalMuscleMechanics2000}.
Moreover, the masses of the cross-bridges are assumed to add no kinetic energy to the macroscopic system. These two assumptions are very reasonable and lead to successful models.
In the following, we want to compute explicitly how the kinetic equations look like in the presence of constraints and give a mathematical quantification for common modelling assumptions.
\index{Kinetic theory!constrained}

We remark that \cref{subsec:kin_char_flow_ode} comprises a short outline of the fundamentals of kinetic theory. For this purpose, it is preferable to use the ODE formulation \cref{eq:dm_newton_eff,eq:dm_constr_ind2}
instead of the DAE formulation \cref{eq:dm_newton_r,eq:dm_newton_qj,eq:dm_constr_ind3}. 
Afterwards, in \cref{subsec:kin_char_flow_dae}, similar steps are applied to the DAE formulation.

\subsection{Partially Kinetic Equations for the ODE Formulation}
\label{subsec:kin_char_flow_ode}

This section outlines the derivation of the partially kinetic equations for the ODE formulation \cref{eq:dm_newton_eff,eq:dm_constr_ind2}. \index{Partially kinetic systems}
We obtain the kinetic equations by the formal mean-field limit. In \cref{sec:mean_field_limit}, this approach is justified by a rigorous estimate which proves that the solutions of the ODE formulation
convergence under the scaling assumptions of the mean-field limit to the solutions of the partially kinetic equations.
There are different techniques to compute the mean-field limit for ODE models. 
The dual pairing between measures and smooth test functions is a powerful formalism to derive the mean-field PDE in its weak formulation \cite{jabinReviewMeanField2014}. However, we will motivate the mean-field limit as a generalisation of the strong law of large numbers, which is a less abstract approach. The derivation follows these steps:
\begin{enumerate}[label={\textbf{Step \arabic*:}},ref={step \arabic*},leftmargin=*]
\item Introduction of a scaling factor, \label{it:ind1_scaling_factor}
\item Introduction of a measure $\pmst$ to describe the statical distribution of the cross-bridge extensions,
\item Derivation of the mean-field characteristic flow equations, which govern the evolution of the cross-bridge distribution $\pmst$,
\item Derivation of a kinetic balance law for the macroscopic system.  \label{it:ind1_kin_eff_law}
\end{enumerate}

Step 1:
As the number of cross-bridges is large, we want to study the limit of infinitely many cross-bridges $N \to \infty$.
A naive limit $N \to \infty$ leads to a system with infinitely many identical linear springs. Such a system is either in equilibrium or entails an infinite force.
The force term in \cref{eq:dm_newton_eff}
\[
\Feffn[\N](r,\qd_1,\dots,\qd_\N) = -\kappar r - \cDphi^T \sum_{j=1}^N \kappaq \qd_j
\]
will either be divergent or the cross-bridge extensions form a zero sequence.
Therefore, the naive limit is mathematically and physically unreasonable since it describes either states close to equilibrium or states with infinite energy.

The scaling assumption of the mean-field limit is that 
the force $\sum_{j=1}^N \Fq(\qd_j)$ is replaced by the mean-field force $\frac{1}{N} \sum_{j=1}^N \Fq(\qd_j)$. \index{Scaling!mean-field}
Therefore, while increasing the number of cross-bridges, we scale the mass and force of each cross-bridge, such that the total energy remains constant.
To maintain the right ratio between the macroscopic system and the cross-bridges, we add another factor $\Nreal$, which denotes the realistic numbers of cross-bridges.
Hence, we apply the following scaling to the ODE formulation \cref{eq:dm_newton_eff,eq:dm_constr_ind2}
\begin{align}
\tilde \Mq \coloneqq \frac{\Nreal}{\N} \Mq \quad \text{and} \quad \tilde \Fq(\qd_j) \coloneqq \frac{\Nreal}{\N} \Fq(\qd_j) = -\frac{\Nreal}{\N} \kappaq \qd_j.
\label{eq:scaling_ode_level}
\end{align}

After this modification, \cref{eq:dm_newton_eff,eq:dm_constr_ind2} take the form
\begin{align}
 \big( \Mr + \overbrace{\frac{\Nreal}{\N}\sum_{i=1}^\N \cDphi^T\Mq \cDphi }^{\eqqcolon M^{(N)}_\text{mean}}  \big) \ddot r &= -\kappar r +  \overbrace{ \frac{\Nreal }{\N} \sum_{i=1}^\N \cDphi^T \kappaq \qd_i}^{\eqqcolon F^{(\N)}_\text{mean}(r,\qd_1,\dots,\qd_\N)},
 \label{eq:dm_scaled_newton_eff}\\
  \dot \qd_j &= - \cDphi \dot r.
  \label{eq:dm_scaled_constr_ind2}
\end{align}
For the mathematical discussion, we might assume without loss of generality $\Nreal = 1$,
which is a typical simplification in kinetic theory \cite{golseDynamicsLargeParticle2016,jabinReviewMeanField2014}.
However, for partially kinetic systems, the correct ratio between masses and forces of both systems is relevant. Therefore, in contrast to the classical case, different values of $\Nreal$ change the properties of the kinetic equations.

Step 2:
A key observation in equation \cref{eq:dm_scaled_newton_eff} is that only the mean value of the cross-bridge masses $M^{(N)}_\text{mean}$ and the mean value of the cross-bridge forces $F^{(\N)}_\text{mean}$ are relevant.
In other words, we are just interested in the statistics of $(\qd_1,\dots,\qd_N)$ but not in particular states of single cross-bridges. This observation motivates the use of a probability measure to quantify the distribution of the cross-bridges.

We will use the following notations from measure theory with notation as in \cite{golseDynamicsLargeParticle2016}: The Borel $\sigma$-algebra on $\Rq$ is denoted by $\BRq$ and the corresponding
space of probability measures on $\Rq$ is $\Cal P(\Rq)$.
The space of probability measures $\mu \in \Cal P(\Rq)$ with finite first moments $\int_\Rq \norm{\qc} \dq \mu < \infty$ is denoted by $\Cal P^{1}(\Rq)$.

We assume that for each fixed time $t$, there is a probability measure $\pmst \in \PRq{1}$, such that the cross-bridge extensions $\qd_j(t)$ are independent and identically distributed random variables with probability law $\pmst$. We use the notation
\[
\qd_j(t) \sim \pmst \quad \vcentcolon\Leftrightarrow \quad \mathbb P( \qd_j(t) \in A ) = \pmst(A), \quad \text{for all~} A \in \BRq. 
\]
We call $\pmst$ the cross-bridge distribution.\index{Cross-bridges!distribution}\footnote{It is not trivial to argue why \emph{all} cross-bridges $\qd_j$ are well described by \emph{one} common probability measure $\pmst$. This property is related to the concept of \emph{propagation of chaos} \cite{jabinReviewMeanField2014}.
The mean-field limit, which generalised the strong law of large numbers, is one possibility to overcome this issue.}

Step 3: To characterise the evolution of $\pmst$, we will now define the characteristic flow. \index{Characteristic flow}
We assume that just the initial cross-bridge distribution $\pmsin \in \Cal P^1(\Rq)$ is known, i.e.
\[
\qdi_j \sim \pmsin, \quad \text{for all~} j = 1,2,\dots.
\]
We interpret the velocity constraint $\dot \qd_j = -\cDphi \dot r$ \cref{eq:dm_constr_ind2} as a first order differential equation and denote its flow by $\cflow(t,\qi)$. Hence,  $\cflow(t,\qi)$ satisfies for all $\qi \in \Rq$
\begin{align}
\dot \cflow(t,\qi) &= - \cDphi \dot r(t), 
\label{eq:muscle_cflow}\\
\cflow(0,\qi) &= \qi. 
\label{eq:muscle_cflow_ini}
\end{align}
In the setting of \cref{eq:dm_scaled_newton_eff,eq:dm_scaled_constr_ind2}, the discrete cross-bridge states satisfy
\begin{align}
\qd_j(t) = \cflow(t,\qdi_j).
\label{eq:relation_cflow_qdj}
\end{align}

\begin{figure}[h]
	\centering
	\includegraphics[width=0.5\textwidth]{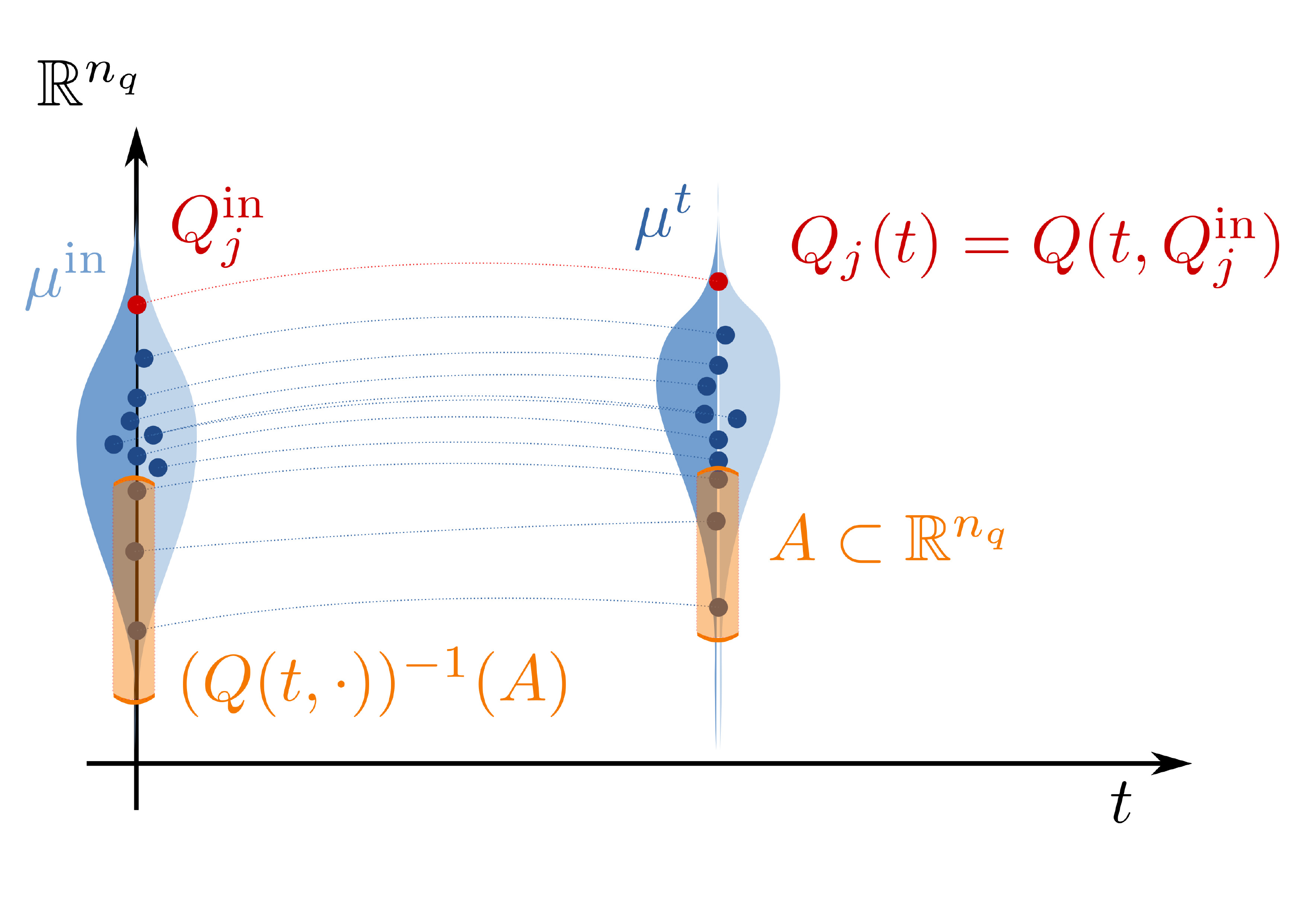}
	\caption{The flow of the particles does also induce a transformation of the initial particle measure $\pmsin$.}
	\label{fig:push_forward}
\end{figure}

Since the cross-bridge extensions follow the characteristic flow of $\cflow(t,\cdot)$, the distribution of cross-bridges is also determined by the characteristic flow.
More precisely, the cross-bridge distribution $\pmst$ is the transformation of the initial cross-bridge distribution $\pmsin$ under the flow $\cflow(t,\cdot) : \Rq \to \Rq$. This transformation is visualized in \cref{fig:push_forward}.
To measure how many cross-bridges have extensions in $A \in \BRq$, we can count how many cross-bridges have an initial extension in $\left(\cflow(t,\cdot)\right)^{-1}(A)$, i.e.
\[
\qd_j(t) \in A \quad \Leftrightarrow \quad \qdi_j \in \left(\cflow(t,\cdot)\right)^{-1}(A).
\]
This relation characterises the pushforward of a measure \index{Pushforward!measure} \cite{golseDynamicsLargeParticle2016,jabinReviewMeanField2014}.
For a map $\varphi : \Rq \to \Rq$, the pushforward of $\pmsin$ under $\varphi$ is defined as
\[
\push{\varphi}{\pmsin}(A) \coloneqq \pmsin(\varphi^{-1}(A)), \quad \text{for all~} A \in \BRq.
\]
Applied to our situation, with $\varphi = \cflow(t,\cdot)$, the cross-bridge distribution at time $t$ is the pushforward of $\pmsin$ under $\cflow(t,\cdot)$. Therefore, the evolution of $\pmst$ is characterised by
\begin{align}
\pmst \coloneqq \push{\cflow(t,\cdot)}{\pmsin}.
\label{eq:muscle_pf_first}
\end{align}

Step 4: Our goal is to approximate the limit $N \to \infty$ of \cref{eq:dm_scaled_newton_eff} by an expression depending on $\pmst$. Consequently, we  continue with computing $M_\text{mean}^{(\N)}$ and $F_\text{mean}^{(\N)}$ \cref{eq:dm_scaled_newton_eff} in the limit $\N \to \infty$.
Now, we make use of the assumption that all cross-bridges are independent and identically distributed with law $\qd_j(t) \sim \pmst$.
Application of the \emph{strong law of large numbers} \index{Strong law of large numbers} \cite[Section 5.3]{klenkeProbabilityTheoryComprehensive2008} yields 
\begin{align}
\lim_{N \to \infty} F_\text{mean}^{(\N)}(t)
&= \Nreal \lim_{N \to \infty} \cDphi^T \kappaq  \frac{1}{\N} \sum_{i=1}^\N \qd_i(t) \notag \\
&= \Nreal \cDphi^T \kappaq  \mathbb{E}\left[ \qd_1(t) \right] \quad \text{almost surely}.
\label{eq:strong_law_large_numbers_mean_field}
\end{align}
The sum converges almost surely (a.s.), i.e. with probability $1$ the last equality holds. 

Using $\E{\qd_1(t)} = \int_\Rq \qc \dq \pmst$, we define the mean-field force as
\begin{align}
f_{\text{mean}}(\pmst) \coloneqq \Nreal \cDphi^T \kappaq  \int_\Rq \qc \dq \pmst.
\label{eq:mf_force_def}
\end{align}
Due to \cref{eq:strong_law_large_numbers_mean_field}, the mean-field force satisfies 
\[
f_{\text{mean}}(\pmst) = \lim_{N \to \infty} F_\text{mean}^{(\N)}(t) \quad a.s.
\]

Similarly, the mean-field mass is 
\begin{align}
m_{\text{mean}}(\pmst) &\coloneqq \Nreal \int_\Rq \cDphi^T \Mq \cDphi \dq \pmst 
\label{eq:mf_mass_def}\\
&= \Nreal \cDphi^T \Mq \cDphi \int_\Rq \dq \pmst = \Nreal  \cDphi^T \Mq \cDphi.\notag
\end{align}

In the setting of our model, the mean-field mass is constant, since we only consider linear constraints.
We remark that for nonlinear constraints, the term $\cDphi$ would depend on $\qc$ in general, 
which calls for a more profound analysis.

With \cref{eq:mf_force_def,eq:mf_mass_def}, we obtain the kinetic formulation of the effective balance law \cref{eq:dm_scaled_newton_eff} as
\begin{align}
\underbrace{ \left( \Mr + \Nreal \cDphi^T \Mq \cDphi  \right) }_{\eqqcolon \meff} \ddot r = \underbrace{-\kappar r + \Nreal \cDphi^T \kappaq \int_\Rq \qc \dq \pmst }_{\eqqcolon \feff(r,\pmst)}. 
\label{eq:muscle_newton_eff_int}
\end{align}

Finally, the kinetic description of the ODE formulation \cref{eq:dm_newton_eff,eq:dm_constr_ind1} is given by the effective balance law \cref{eq:muscle_newton_eff_int}, the characteristic flow equations \cref{eq:muscle_cflow} and the pushforward relation \cref{eq:muscle_pf_first}.

The \emph{partially kinetic model} for muscle cells with attached cross-bridges 
\index{Cross-bridges!attached} \index{Partially kinetic systems!ODE form} is given by \cref{eq:muscle_newton_eff_int,eq:muscle_cflow,eq:muscle_cflow_ini,eq:muscle_pf_first}, summarised as
\begin{align}
\left( \Mr + \Nreal \int_\Rq \cDphi^T \Mq \cDphi \dq \pmst \right) \ddot r &=-\kappar r + \Nreal \cDphi^T \kappaq \int_\Rq  \qc \dq \pmst, 
\label{eq:muscle_newton_eff}\\
\dot \cflow(\cdot,\qi) &= - \cDphi \dot r, \quad \text{for all~} \qi \in \Rq, 
\label{eq:muscle_flow}\\
\pmst &\coloneqq \push{\cflow(t,\cdot)}{\pmsin}
\label{eq:muscle_pushforward}
\end{align}
with initial conditions
\[
r(0) = \ri \in \Rr, \quad \dot r(0) = \si \in \Rr \quad \text{and} \quad \cflow(0,\qi) = \qi
\]
and initial cross-bridge distribution $\pmsin \in \PRq{1}$.
In kinetic theory, systems of the form \cref{eq:muscle_flow,eq:muscle_pushforward}
are called the \emph{mean-field characteristic flow equations}\index{Characteristic flow!mean-field} \cite[Section 1.3]{golseDynamicsLargeParticle2016} since $\cflow(t,\cdot)$ describes the flow of the cross-bridge distribution in the presence of the mean-field forces.
For that reason, we call the model \cref{eq:muscle_newton_eff,eq:muscle_flow,eq:muscle_pushforward} \emph{partially kinetic}, since it combines an effective balance law for the macroscopic system \cref{eq:muscle_newton_eff}  and the \emph{mean-field characteristic flow equations} \cref{eq:muscle_flow,eq:muscle_pushforward}.
 
\subsection{Consistency between the ODE Formulation and the Partially Kinetic Equations.}
 
The cross-bridge distribution $\pmst$ is usually modelled as a continuous measure. However, inserting an empirical measure for the initial state $\pmsin$ yields a consistency check of  \cref{eq:muscle_newton_eff,eq:muscle_flow,eq:muscle_pushforward}.
We define the empirical measure \index{Empirical measure} as
\begin{align}
\empn_{\qd_1,\dots,\qd_N} \coloneqq \frac{1}{N} \sum_{j=1}^N \delta_{\qd_j} \in \PRq{1}, \notag
\end{align}
where $\delta_{\qd_j}$ denotes the Dirac measure which assigns unit mass to the position $\qd_j \in \Rq$. This measure allows us to treat the discrete system \cref{eq:dm_newton_eff,eq:dm_constr_ind2} as a special case of  \cref{eq:muscle_newton_eff,eq:muscle_flow,eq:muscle_pushforward}.

\begin{lemma}[Consistency with the ODE formulation]
	\label{lem:insert_emperical}
	
	For a solution $(r(t), \qd_1(t), \dots, \qd_\N(t))$ of \cref{eq:dm_newton_eff,eq:dm_constr_ind2},
	we define 
	\[
	\pmst = \empn_{\qd_1(t),\dots,\qd_N(t)}.
	\]
	Then $(r(t),\pmst)$ is a solution of \cref{eq:muscle_newton_eff,eq:muscle_flow,eq:muscle_pushforward} with $\Nreal \coloneqq \N$ and initial conditions $r(0) = \ri, \dot r(0) = \si$ and $\pmsin = \empn_{\qdi_1,\dots,\qdi_N}$.
	
\end{lemma}
\begin{proof}
Let $(r(t), \qd_1(t), \dots, \qd_\N(t))$ be a solution of \cref{eq:dm_newton_eff,eq:dm_constr_ind2} with the corresponding initial conditions $r(0) = \ri, \dot r(0) = \si$ and $\qd_j(0) = \qi_j$ for all $j=1,\dots,\N$.

1. We define the characteristic flow as 
\[
\cflow(t,\qi) \coloneqq -\cDphi (r(t) - \ri) + \qi,
\]
which is the integral of \cref{eq:muscle_flow} and hence a solution of \cref{eq:muscle_flow} for all $\qi \in \Rq$.

2. Since $\qd_1(t), \dots, \qd_\N(t)$ satisfy \cref{eq:dm_constr_ind3}
\[
\qd_j(t) = -\cDphi (r(t) - \ri) + \qdi_j,
\]
we obtain $\cflow(t,\qdi_j) = \qd_j(t)$, which yields
\begin{align*}
\push{\cflow(t,\cdot)}{\pmsin} 
&= \push{\cflow(t,\cdot)}{\empn_{\qdi_1,\dots,\qdi_\N}} \\
&= \empn_{\cflow(t,\qdi_1),\dots,\cflow(t,\qdi_\N)} \\
&= \empn_{\qd_1(t),\dots,\qd_\N(t)}
= \pmst.
\end{align*}
As a result, $\pmst$ satisfies \cref{eq:muscle_pushforward}.

3. Finally, we insert $\pmst \coloneqq \empn_{\qd_1(t),\dots,\qd_\N(t)}$ into \cref{eq:muscle_newton_eff} and compute
\begin{align}
\left( \Mr + \Nreal \int_\Rq \cDphi^T \Mq \cDphi \dq {\empn_{\qd_1(t),\dots,\qd_\N(t)}} \right) \ddot r(t) &=-\kappar r(t)
\notag \\
& \quad + \Nreal \cDphi^T \kappaq \int_\Rq  \qc \dq  {\empn_{\qd_1(t),\dots,\qd_\N(t)}} \notag \\
\Leftrightarrow \left( \Mr + \Nreal \frac{1}{N} \sum_{j=1}^N \cDphi^T \Mq \cDphi \right) 
\ddot r(t) &=-\kappar r(t) + \Nreal \frac{1}{N} \sum_{j=1}^N \cDphi^T \kappaq \qd_j(t) \notag \\
\Leftrightarrow \quad \left( \Mr +  \Nreal \cDphi^T \Mq \cDphi \right) 
\ddot r(t) &=-\kappar r(t) + \frac{\Nreal}{N} \sum_{j=1}^N \cDphi^T \kappaq \qd_j(t).
\label{eq:consisteny_emp_inserted}
\end{align}
For $\Nreal \coloneqq \N$, the last line \cref{eq:consisteny_emp_inserted} is exactly \cref{eq:dm_newton_eff}. Therefore, $(r(t),\pmst)$ solve \cref{eq:muscle_newton_eff}, which concludes the proof.
\qed
\end{proof}

\cref{lem:insert_emperical} shows that the ODE formulation \cref{eq:dm_newton_eff,eq:dm_constr_ind1} is a special case of the partially kinetic system \cref{eq:muscle_newton_eff,eq:muscle_flow,eq:muscle_pushforward} with an empirical initial measure $\pmsin =\empn_{\qdi_1,\dots,\qdi_N}$.
The consistency check from \cref{lem:insert_emperical} does not prove anything for the limit $\N \to \infty$, but it relates the limit $\N \to \infty$ to the stability of the partially kinetic system with respect to initial data. \index{Stability} We will prove the stability of linear partially kinetic systems with respect to perturbation in the initial data in \cref{sec:mean_field_limit}.

\subsection{Partially Kinetic Mean-Field PDE}
\label{subsec:muscle_mf_pde}

The mean-field characteristic flow equations \cref{eq:muscle_flow,eq:muscle_pushforward} are too complex for direct numerical simulation, as they contain an infinite family of differential equations.
The method of characteristics \index{Method of characteristics} allows us to relate the family of ODEs \cref{eq:muscle_flow,eq:muscle_pushforward}
to a first order partial differential equation (PDE). \index{PDE} \index{Differential equation!partial|see {PDE}} The resulting PDE is called the \emph{mean-field PDE}. \index{PDE!mean-field}\index{Mean-field limit!PDE}

To derive the mean-field PDE, we assume that $\pmsin$ has a probability density, i.e. there exists
a function $\pdens(t,\qc)$ such that
\[
\pdens(t,\qcp) \dif \qcp = \dqp{\pmst} \quad \text{for all~} t \in [0,\infty), \,\, \qcp \in \Rq
\]
where $\dif \qcp$ denotes the Lebesgue measure on $\Rq$ with variable $\qcp$.
Then, the pushforward relation \cref{eq:muscle_pushforward} implies that $u(t,\qcp)$ is constant along the characteristic curves 
\[
t \mapsto \cflow(t,\qi).
\]
Using this invariance, we can compute
\begin{align}
0 &= \dod{\pdens(t,\cflow(t,\qi))}{t} 
= \dpd{\pdens}{t} + 
\dpd{\pdens}{\qc} \dot \cflow \notag\\
\Leftrightarrow \quad 0 &= \dpd{\pdens}{t} - 
\dpd{\pdens}{\qc} \cDphi \dot r.
\label{eq:muscle_transport_eq_first}
\end{align}
The transport equation \index{PDE!transport} \cref{eq:muscle_transport_eq_first} is the mean-field PDE for the ODE formulation \cref{eq:dm_newton_eff,eq:dm_constr_ind2} and the equations \cref{eq:muscle_transport_eq_first,eq:muscle_newton_eff} are another kinetic description for attached cross-bridges.
The possibility to derive a mean-field PDE is the main advantage of the ODE formulation.
In the literature \cite{keenerMathematicalPhysiology2009,howardMechanicsMotorProteins2001}, 
the foundational model for cross-bridge dynamics is the transport equation as in \cref{eq:muscle_transport_eq_first} with additional source terms. In \cref{subsec:compare_estab_models}, the relation between the simplified model of attached cross-bridges and more realistic models is outlined.

A numerical simulation of \cref{eq:muscle_transport_eq_first,eq:muscle_newton_eff} is presented in \cref{fig:linear_meso}.

In contrast to the ODE/DAE simulation, the computational complexity of the transport equation does not increase in complexity for different values of $\Nreal$. 
For more details on the used numerical methods, we refer to \cref{sec:numerical_examples}.
The simulation results are not surprising, and the results fit well to the simulation of the ODE/DAE formulation, as visualised in \cref{fig:linear_meso_compare}.

However, \cref{fig:linear_meso_compare} can be misleading: 
If we denote the solution of the partially kinetic system as $r^{\mathrm{kin}}(t;\pmsin)$ and the solution of the DAE formulation as 
$r^{\mathrm{DAE}}(t;\qdi_1,\dots,\qdi_\N)$, then
the relation is given by 
\begin{align}
\lim_{\N \to \infty} \E{r^{\mathrm{DAE}}(t;\qdi_1,\dots,\qdi_\N)} = r^{\mathrm{kin}}(t;\pmsin),
\label{eq:mf_limit_heavy}
\end{align}
where $\qdi_j \sim \pmsin$ are independent and identically distributed random variables.
Hence, instead of comparing single trajectories $r^{\mathrm{DAE}}(t)$ and $r^{\mathrm{kin}}(t)$,
we need to compare the mean-trajectory $\E{r^{\mathrm{DAE}}(t)}$ with $r^{\mathrm{kin}}(t)$. \cref{subsec:mean_field_numerical} gives a numerical validation of \cref{eq:mf_limit_heavy} and presents numerical evidence for the mean-field limit in \cref{fig:linear_mc,fig:linear_mf_error}.

\begin{figure}[h]
	\centering
	\includegraphics[width=0.45\textwidth]{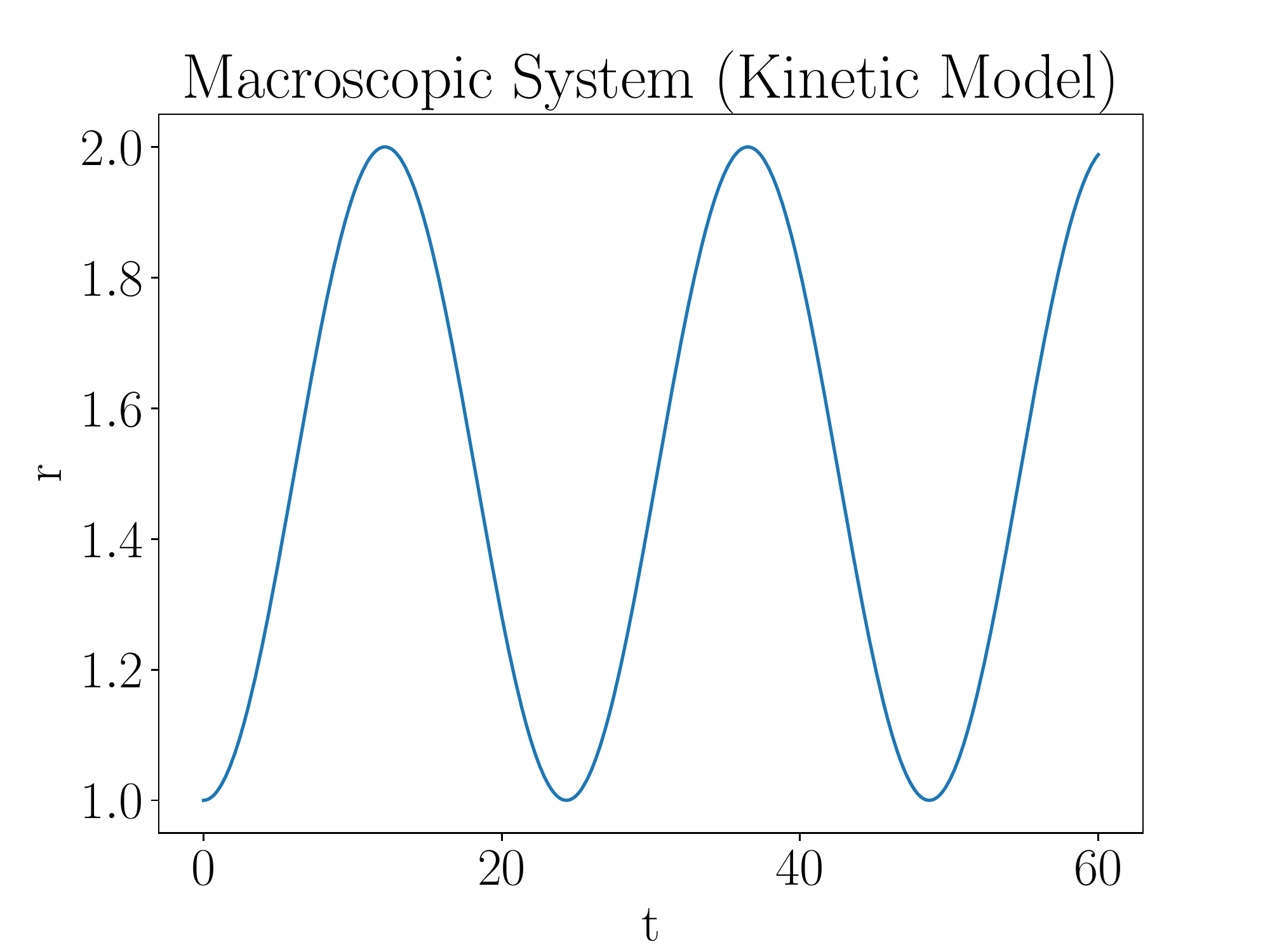}
	\hfill
	\includegraphics[width=0.45\textwidth]{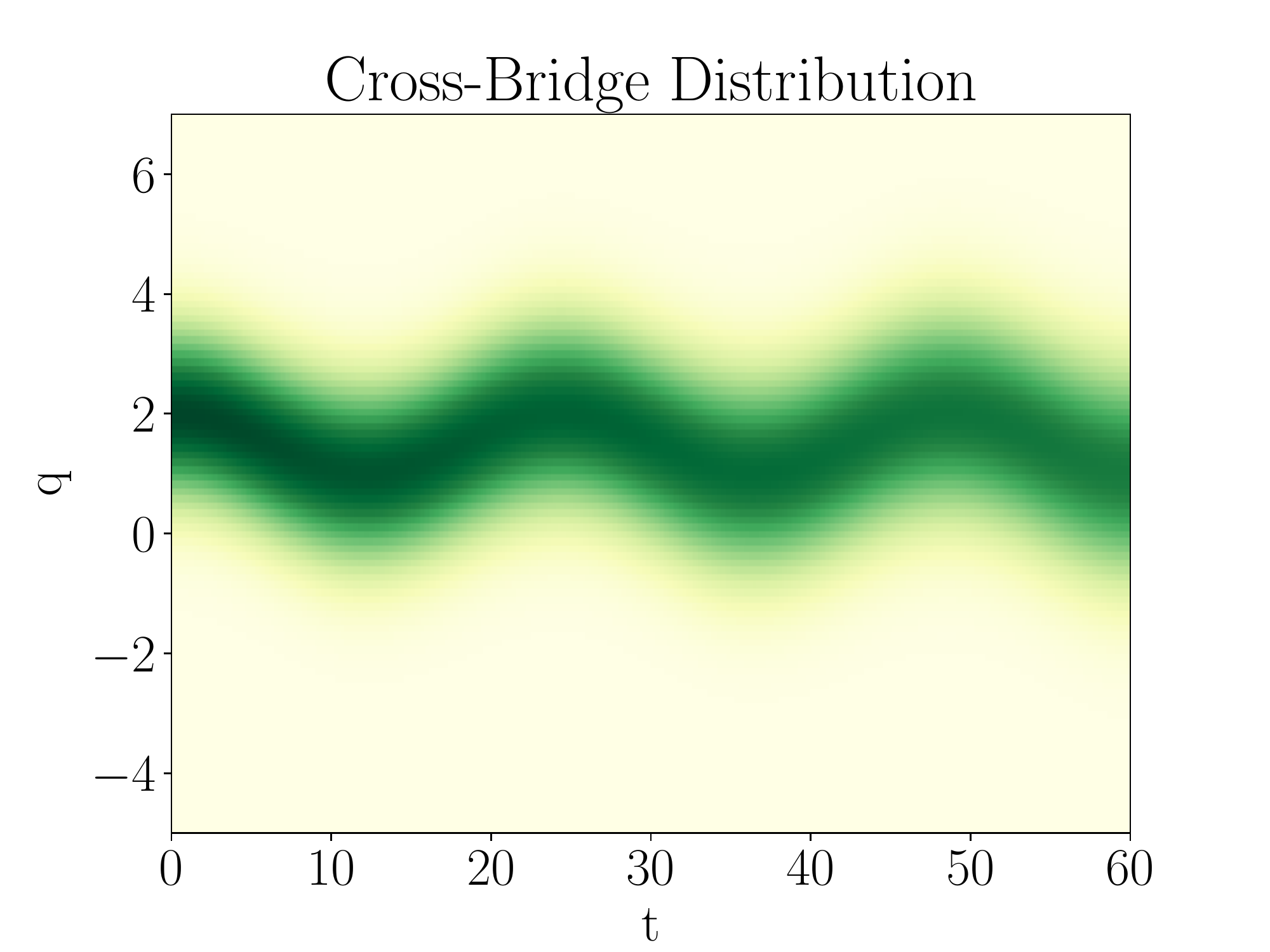}
	\caption{Evolution of the macroscopic system (left) and the cross-bridge  distribution (right). The colour intensity represents the cross-bridge density $\pdens(t,\qc)$.}
	\label{fig:linear_meso}
\end{figure}

\begin{figure}[h]
	\centering
	\includegraphics[width=0.45\textwidth]{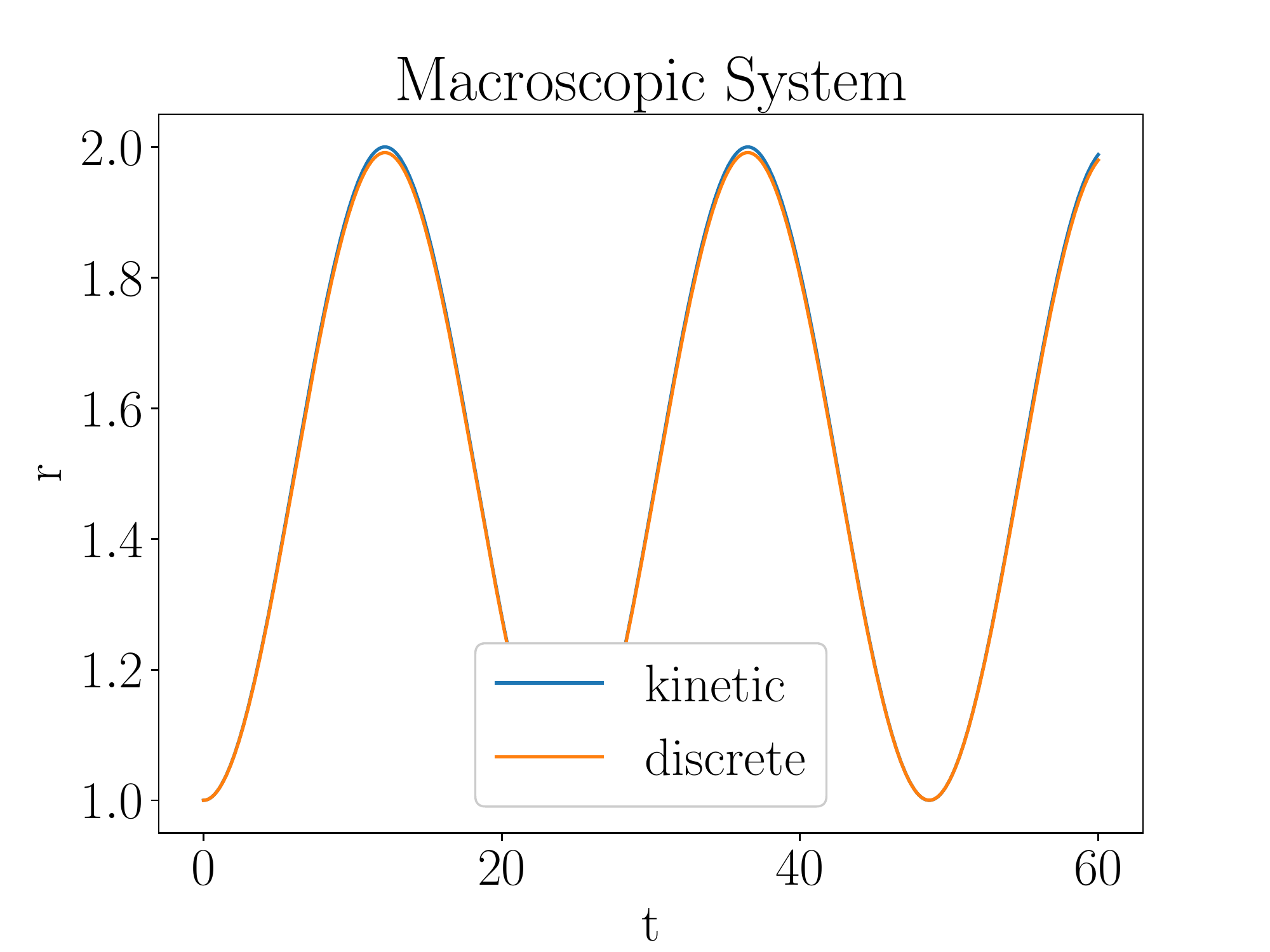}
	\hfill
	\includegraphics[width=0.45\textwidth]{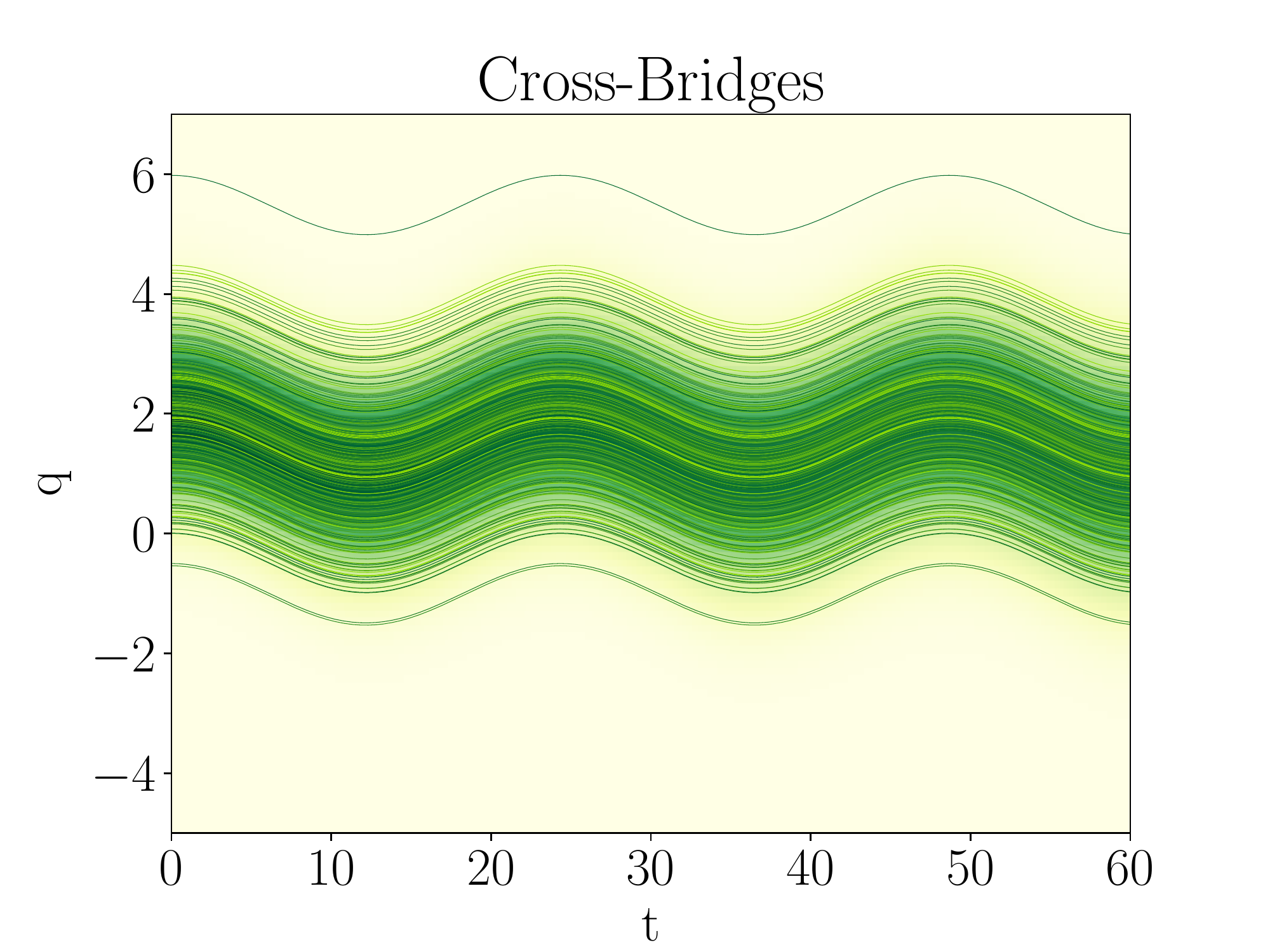}
	\caption{Comparison of the discrete simulation \cref{fig:linear_disc} and the corresponding kinetic trajectory from \cref{fig:linear_meso}.
		The trajectory of the macroscopic system with $250$ cross-bridges is well approximated by the corresponding mean-field equation.}
	\label{fig:linear_meso_compare}
\end{figure}

\subsection{Partially Kinetic Equations for the DAE Formulation}
\label{subsec:kin_char_flow_dae}

To analyse the influence of index reduction, we demonstrate how the mean-field limit applies directly to the DAE \cref{eq:dm_newton_r,eq:dm_newton_qj,eq:dm_constr_ind3} without prior index reduction. 
Here, we need to generalise the Lagrangian multipliers \index{Lagrangian multipliers} to fit into the kinetic framework.
The resulting characteristic flow has to satisfy the algebraic constraint \cref{eq:dm_constr_ind3},  which leads to the new concept of constrained characteristic flows. 
\index{Characteristic flow!constrained mean-field}
Since the constraint \cref{eq:dm_constr_ind3} is uniform for all $j\in\{1,\dots,N\}$, the notation of a differentiability index \cite{BrCP96,hairerNumericalSolutionDifferentialAlgebraic1989} carries over to constrained characteristic flows. \index{Partially kinetic systems!index} After two index reduction steps and elimination of multipliers, as in \cref{subsec:dm_index_reduction}, the constrained characteristic flow equations transform into the (unconstrained) characteristic flow equations \cref{eq:muscle_flow}.  In abstract terms, the index reduction and the mean-field limit commute, as summarised in \cref{fig:different_derivations}.

\begin{figure}[h]
	\centering
	\begin{tikzcd} [row sep=0.8cm, column sep=2.5cm]
		\fbox{\parbox{3cm}{\centering Differential-Algebraic Equation (index 3)}}
		\arrow{r}{\text{formal mean-field limit}}[swap]{\text{(see \cref{subsec:kin_char_flow_dae})}}
		\arrow{d}{\text{(see \cref{subsec:dm_index_reduction})}}[swap]{\text{index reduction}}
		&
		\fbox{\parbox{3cm}{\centering Constrained Mean-Field Characteristic Flow\\ (index 3)}}
		\arrow{d}{\text{(see \cref{subsec:kin_index_reduction})}}[swap]{\text{index reduction}}
		\\	
		\fbox{\parbox{3cm}{\centering Differential-Algebraic Equation (index 1)}}
		\arrow{d}{\text{(see \cref{subsec:dm_index_reduction})}}[swap]{\text{\parbox{1.4cm}{\centering elimination of multipliers}}}
		\arrow[r,dashrightarrow, "{\text{formal mean-field limit}}"]
		&
		\fbox{\parbox{3cm}{\centering Constrained Mean-Field Characteristic Flow\\ (index 1)}}
		\arrow{d}{\text{(see \cref{subsec:kin_index_reduction})}}[swap]{\text{\parbox{1.4cm}{\centering elimination of multipliers}}}
		\\
		\fbox{\parbox{3cm}{\centering Ordinary Differential Equation}}
		\arrow{r}{\text{formal mean-field limit}}[swap]{\text{(see \cref{subsec:kin_char_flow_ode})}}
		&
		\fbox{\parbox{3cm}{\centering Mean-Field Characteristic Flow}}
	\end{tikzcd}
	\caption{Different paths to derive the mean-field characteristic flow equations for \cref{eq:dm_newton_r,eq:dm_newton_qj,eq:dm_constr_ind3}.}
	\label{fig:different_derivations}
\end{figure}
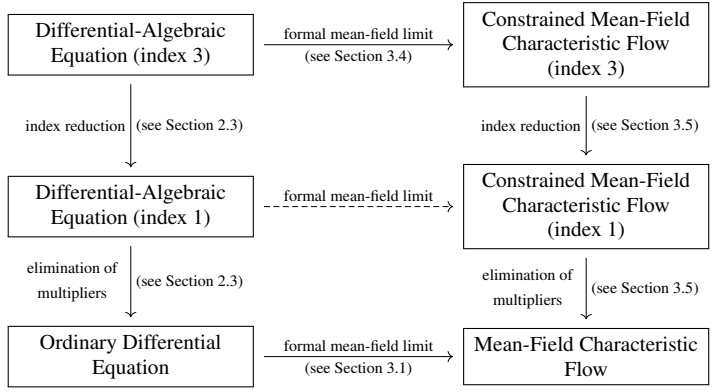

In order to formally derive the kinetic equations for the system of DAEs \cref{eq:dm_newton_r,eq:dm_newton_qj,eq:dm_constr_ind3}, we follow similar steps as in \cref{subsec:kin_char_flow_ode}. 
\begin{enumerate}[label=\textbf{Step \arabic*:}, ref={step \arabic*}, leftmargin=*]
\item Introduction of a scaling factor, \label{it:ind3_scaling_factor} \index{Scaling!mean-field}
\item Introduction of a measure $\pmst$ to describe the statical distribution of the cross-bridge extensions,
\item Derivation of the \emph{constrained} mean-field characteristic flow equations, which govern the evolution of the cross-bridge distribution $\pmst$,
\item Derivation of a kinetic balance law for the macroscopic system.  \label{it:ind3_kin_eff_law}
\end{enumerate}

Step 1: \index{Mean-field limit}
Exactly as in \cref{eq:scaling_ode_level}, we scale the mass and force of each cross-bridge via
\begin{align}
\tilde M_\qc \coloneqq \frac{\Nreal}{\N} \Mq \quad \text{and} \quad \tilde \Fq(\qd_j) \coloneqq \frac{\Nreal}{\N} \Fq(\qd_j) = -\frac{\Nreal}{\N} \kappaq \qd_j,
\end{align}
where $\Nreal$ is the fixed number of cross-bridges that a realistic system would have, and $\N$ is the number of cross-bridges of the systems for the limit process $\N \to \infty$.
After this rescaling, the system of DAEs \cref{eq:dm_newton_r,eq:dm_newton_qj,eq:dm_constr_ind3} is
\begin{align}
\Mr \ddot r &= - \kappar r - \sum_{i=1}^{\N} \cDphi^T \tilde \lambda_i , 
\label{eq:scaled_newton_r}
\\
\frac{\Nreal}{\N} \Mq \ddot \qd_j &= - \frac{\Nreal}{\N} \kappaq \qd_j - \tilde \lambda_j \quad \text{for~} j=1,\dots,\N, 
\label{eq:scaled_newton_qj}
\\
\qd_j + \cDphi r  &= \qdi_j + \cDphi \ri  \quad \text{for~} j=1,\dots,\N
\label{eq:scaled_constr_ind3}
\end{align}
where $\tilde \lambda_i \in \Rq$ are the Lagrangian multipliers of the scaled system \cref{eq:scaled_newton_r,eq:scaled_newton_qj,eq:scaled_constr_ind3}.

Step 2:
Precisely as in \cref{subsec:kin_char_flow_ode}, we assume that the cross-bridge distribution $\pmst \in \PRq{1}$ exists. In the following, we want to derive a law for the evolution of $\pmst$. 
The initial cross-bridge distribution is given by $\pmsin \in \PRq{1}$.

Step 3:
In contrast to the effective balance law \cref{eq:dm_newton_eff}, the balance law of the macroscopic system \cref{eq:scaled_newton_r} contains Lagrangian multipliers.
The value of the multiplier $\tilde \lambda_j$ in \cref{eq:scaled_newton_qj} represents a force acting on the $j$th cross-bridge. 
In a statistical description, there are no individually labelled cross-bridges any more. 
Instead, the characteristic flow $\cflow(t,\qi)$ tracks the dynamics of cross-bridges with initial condition $\qi \in \Rq$.
Hence, instead of one multiplier $\tilde \lambda_j(t)$ per cross-bridge, the kinetic description requires  a multiplier such as $\tilde \lambda(t,\cflow(t,\qi))$ for each initial state.

The characteristic flow $\cflow(t,\qi)$ and the generalised Lagrangian multipliers $\tilde \lambda(t,\qc)$ are solutions of the following family of DAEs
\begin{align}
\frac{\Nreal}{\N} \Mq \ddot \cflow(t,\qi) &= - \frac{\Nreal}{\N} \kappaq \cflow(t,\qi) - \tilde \lambda(t,\cflow(t,\qi)) \quad &\text{for all~} \qi \in \Rq, 
\label{eq:kin_newton_qj_unscaled}
\\
\cflow(t,\qi) + \cDphi r &= \qi + G_r \ri   &\text{for all~} \qi \in \Rq.
\label{eq:kin_constr_ind3_unscaled}
\end{align}
However, this formulation is not satisfying since we can formally compute 
\begin{align}
\tilde \lambda(t,\cflow(t,\qi)) &= -\frac{\Nreal}{\N} \left( \Mq \ddot \cflow(t,\qi) - \frac{\Nreal}{\N} \kappaq \cflow(t,\qi) \right) \\
&\to 0 \quad \text{(formally) for~} \N \to \infty.
\end{align}
As a result, the quantities of interest in the mean-field limit are not the Lagrangian multipliers $\tilde \lambda(t,\cflow(t,\qi))$ but the mean-field Lagrangian multipliers \index{Lagrangian multipliers!mean-field}
\[
\lambda_{\mathrm{mf}}(t,\cflow(t,\qi)) \coloneqq \frac{\N}{\Nreal} \tilde \lambda(t,\cflow(t,\qi)) \quad \text{for all~} \qi \in \Rq.
\]
With this definition, the system \cref{eq:kin_newton_qj_unscaled,eq:kin_constr_ind3_unscaled} reads
\begin{align}
\Mq \ddot \cflow(t,\qi) &= -\kappaq \cflow(t,\qi) - \lambda_{\mathrm{mf}}(t,\cflow(t,\qi)) \quad &\text{for all~} \qi \in \Rq, 
\label{eq:kin_newton_qj_first}
\\
\cflow(t,\qi) + \cDphi r &= \qi + G_r \ri   &\text{for all~} \qi \in \Rq.
\label{eq:kin_constr_ind3_first}
\end{align}
We refer to \cref{eq:kin_newton_qj_first,eq:kin_constr_ind3_first} as the \emph{constrained mean-field characteristic flow equations}.
For the same arguments as in \cref{subsec:kin_char_flow_ode} and \cref{fig:push_forward}, the evolution of the cross-bridge distribution $\pmst$ is the pushforward of $\pmsin$ under the constrained characteristic flow
\begin{align}
\pmst &\coloneqq \push{\cflow(t,\cdot)}{\pmsin}.
\label{eq:kin_pushforward_dae_first}
\end{align}

Step 4:
To replace \cref{eq:scaled_newton_r} by its kinetic counterpart, we apply the scaling $\lambda_j = \frac{N}{\Nreal} \tilde \lambda_j$, to obtain formally
\begin{align}
\Mr \ddot r 
&=-\kappar r - \sum_{i=1}^\N \cDphi^T \tilde \lambda_i
\notag \\
&= -\kappar r - \frac{\Nreal}{\N} \sum_{i=1}^\N \cDphi^T \lambda_i \notag \\
&= -\kappar r - \Nreal \left( \frac{1}{\N} \sum_{i=1}^\N \cDphi^T \lambda_{\mathrm{mf}}(t,\cflow(t,\qdi_i)) \right) \notag \\
&\to -\kappar r - \Nreal \int_\Rq \cDphi^T \lambda_{\mathrm{mf}}(t,\cflow(t,\qi)) \pmsin \quad \text{(formally) for~}\N \to \infty 
\label{eq:formal_limit_dae} \\
&= -\kappar r - \Nreal \int_\Rq \cDphi^T \lambda_{\mathrm{mf}}(t,\qc) \pmst
\notag
\end{align}
This formal argument yields the kinetic balance law
\begin{align}
\Mr \ddot r &= - \kappar r - \Nreal \int_\Rq \cDphi^T \lambda_{\mathrm{mf}}(t,\qc) \dq \pmst , 
\label{eq:kin_newton_r_first}
\end{align}
We remark that, in contrast to the derivation in \cref{subsec:kin_char_flow_ode}, the strong law of large numbers is not applicable in \cref{eq:formal_limit_dae}.

Finally, we arrive at the \emph{partially kinetic index-3 formulation} \index{Partially kinetic systems!DAE form} of \cref{eq:dm_newton_r,eq:dm_newton_qj,eq:dm_constr_ind3}  which is given by \cref{eq:kin_newton_r_first,eq:kin_newton_qj_first,eq:kin_constr_ind3_first,eq:kin_pushforward_dae_first}. We summarise these equations
\begin{align}
\Mr \ddot r &= - \kappar r - \Nreal \int_\Rq \cDphi^T \lambda_{\mathrm{mf}}(t,\qc) \dq \pmst , 
\label{eq:kin_newton_r} 
\\
\Mq \ddot \cflow(t,\qi) &= -\kappaq \cflow(t,\qi) - \lambda_{\mathrm{mf}}(t,\cflow(t,\qi)) \quad &\text{for all~} \qi \in \Rq, 
\label{eq:kin_newton_qj}
\\
\cflow(t,\qi) + \cDphi r &= \qi + G_r \ri  &\text{for all~} \qi \in \Rq,
\label{eq:kin_constr_ind3}
\\
\pmst &\coloneqq \push{\cflow(t,\cdot)}{\pmsin}
\label{eq:kin_pushforward_dae}
\end{align}
with initial conditions
\begin{align}
r(0) = \ri \in \Rr, \quad \dot r(0) = \si \in \Rr \quad \text{and} \quad \cflow(0,\qi) = \qi \quad \text{for all~} \qi \in \Rq.
\end{align}
Since the derivation is formal, regularity aspects are not considered. Indeed, the integral in \cref{eq:kin_newton_r} might not exist for $\pmsin \in \PRq{1}$.
It could be, that higher regularity of the initial data is necessary to ensure the existence of solutions of the partially kinetic index-3 formulation.
Hence, we do not claim that the system has solutions.

\subsection{Index Reduction and Elimination of the Multipliers for Partially Kinetic Systems}
\label{subsec:kin_index_reduction}

In \cref{subsec:dm_index_reduction}, for every cross-bridge index $j \in \{1,\dots,\N\}$, the same algebraic transformations recast the DAE formulation into the index-1 formulation. \index{Partially kinetic systems!index}
Despite  \cref{eq:kin_newton_qj,eq:kin_constr_ind3} being infinite-dimensional,
the index reduction is possible with the same algebraic transformations as in \cref{subsec:dm_index_reduction} but applied for each $\qi \in \Rq$.
The key assumption here is the uniformity of the constraints, since for each $\qi \in \Rq$ the algebraic constraint in \cref{eq:kin_constr_ind3} just differs by a shift.

The first two time derivatives of \cref{eq:kin_constr_ind3} are
\begin{align}
\dot \cflow(t,\qi) &= -\cDphi \dot r(t),
\label{eq:kin_constr_ind2}\\
\ddot \cflow(t,\qi) &= -\cDphi \ddot r(t).
\label{eq:kin_constr_ind1}
\end{align}
Now, we solve \cref{eq:kin_newton_qj} for the mean-field Lagrangian multipliers
\begin{align*}
\lambda_{\mathrm{mf}}(t,\cflow(t,\qi)) 
&= -\kappaq \cflow(t,\qi) - \Mq \ddot \cflow(t,\qi),
\end{align*}
and with \cref{eq:kin_constr_ind1} we obtain
\begin{align}
\lambda_{\mathrm{mf}}(t,\cflow(t,\qi))  &= -\kappaq \cflow(t,\qi) + \Mq G_r \ddot r(t).
\label{eq:muscle_lambda_formula_kin}
\end{align}
Finally, we substitute \cref{eq:muscle_lambda_formula_kin} into the kinetic balance law \cref{eq:kin_newton_r}
\begin{align}
\Mr \ddot r &= -\kappar r - \Nreal \int_\Rq G_r^T \lambda_{\mathrm{mf}}(t,\qc) \dq \pmst \notag \\ 
&= -\kappar r - \Nreal  \int_\Rq G_r^T (-\kappaq \qc + \Mq G_r \ddot r(t)) \dq \pmst \notag \\
&= -\kappar r + \Nreal  \int_\Rq G_r^T \kappaq \qc \dq \pmst - \left( \int_{\Rq} \Nreal  G_r^T \Mq G_r \dq \pmst  \right) \ddot r(t).
\label{eq:ind_red_eff_newton}
\end{align} 
The resulting equation \cref{eq:ind_red_eff_newton} is exactly the previously derived effective kinetic balance law for the macroscopic system in \cref{eq:muscle_newton_eff}.
The equations \cref{eq:ind_red_eff_newton,eq:kin_constr_ind2,eq:kin_pushforward_dae} are exactly the previously derived partially kinetic equations \cref{eq:muscle_newton_eff,eq:muscle_flow,eq:muscle_pushforward}.
This concludes the claim from the beginning of \cref{subsec:kin_char_flow_dae}: The formal mean-field limit and index-reduction commute for the DAE formulation \cref{eq:dm_newton_r,eq:dm_newton_qj,eq:dm_constr_ind3}.

\section{The Mean-Field Limit for Partially Kinetic Systems}
\label{sec:mean_field_limit}

Until now, all derivations of partially kinetic equations in this article have been formal. In this section, we will give a rigorous link between the discrete cross-bridge dynamics and their kinetic description.
In equation \cref{eq:strong_law_large_numbers_mean_field}, the strong law of large  numbers \index{Strong law of large numbers}\index{Mean-field limit} motivates the use of a mean-field force.
Nevertheless, the strong law of large numbers is not sufficient to show that the kinetic description is a good approximation for systems with many cross-bridges.
The law requires all cross-bridges to be stochastically independent, which is not easy to prove in general. 
Moreover, the argument in \cref{eq:strong_law_large_numbers_mean_field} only applies for a fixed time $t$.
The mean-field limit is a generalization of the strong law of large numbers and yields the required convergence result.

\cref{subsec:mean_field_numerical} motivates the mean-field limit with a numerical evidence.  \cref{subsec:mean_field_dobr} shows that the mean-field limit follows from Dobrushin's stability estimate for partially kinetic systems. Thereafter, \cref{subsec:proof_dobr} proofs the stability estimate.

\subsection{Numerical Evidence of the Mean-Field Convergence}
\label{subsec:mean_field_numerical}

A consequence of the mean-field limit is that the mean-trajectories of the macroscopic system in the discrete formulation tend towards the solution of the partially kinetic description, as explained in \cref{subsec:muscle_mf_pde}. We will now perform a numerical test for this claim.

In \cref{fig:linear_mc}, we perform numerical simulations of the ODE formulation \cref{eq:dm_newton_eff,eq:dm_constr_ind2}
with fixed initial conditions $\ri = 1$
and $\si = 0$.
The statistics of the initial cross-bridge extensions are given by  a normal distribution 
\[
\dif \pmsin(\qc) = \frac{1}{\sqrt{2\pi \sigma^2}} \exp\left({ - \frac{(\qc - m)^2}{2  \sigma} }\right) \dif \qc,
\]
with mean $m = -2$ and variance $\sigma^2 = 1$.
As initial data for the ODE formulation, we sample $n_{\text{samples}} = 100$ many initial cross-bridge extensions
$\qdi_{i,k},\dots,\qdi_{\N,m} \in \Rq$, with law
\[
\qdi_{i,k} \sim \pmsin \quad \text{for $i\in\{1,\dots,\N\}$, $k \in \{1,\dots,n_{\text{samples}}\}.$}
\]
The sampling  yields $n_{\text{samples}}$ different initial conditions and therefore, $n_{\text{samples}}$ trajectories of the macroscopic system and the cross-bridges.
The trajectories of the macroscopic system are plotted in \cref{fig:linear_mc}.

 To quantify the distance of single trajectories from the mean-trajectory,
 we estimate the variance of $r(t)$ with respect to the randomly sampled initial conditions. It turns out that the variance reduces asymptotically as fast as $\frac{1}{\N}$, which is displayed in \cref{fig:linear_mc}.
Therefore, even single trajectories are close to the mean-trajectory.

\begin{figure}[h]
	\centering
	\includegraphics[width=0.45\textwidth]{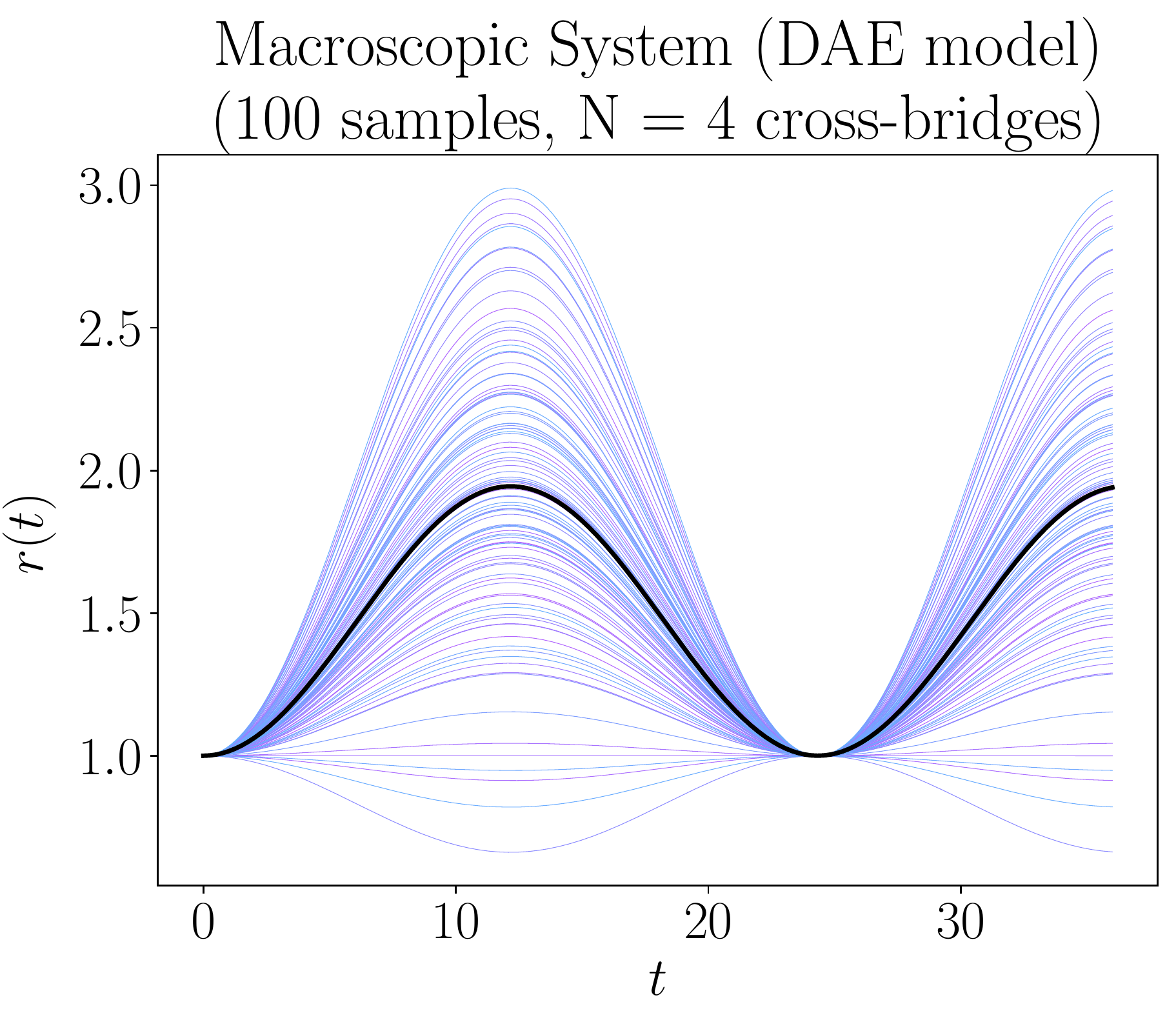}
	\hfill
	\includegraphics[width=0.45\textwidth]{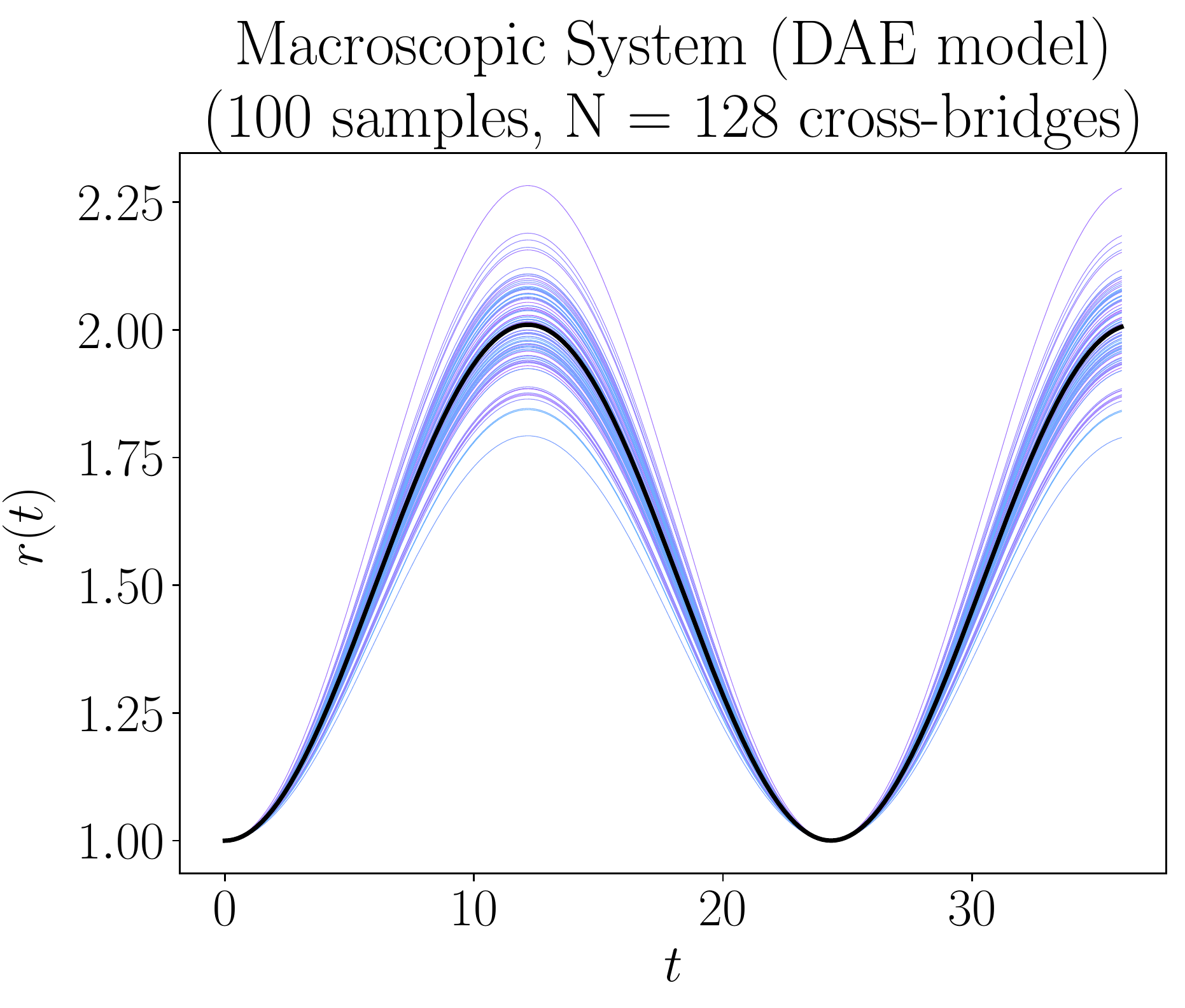}
	\\
	\includegraphics[width=0.45\textwidth]{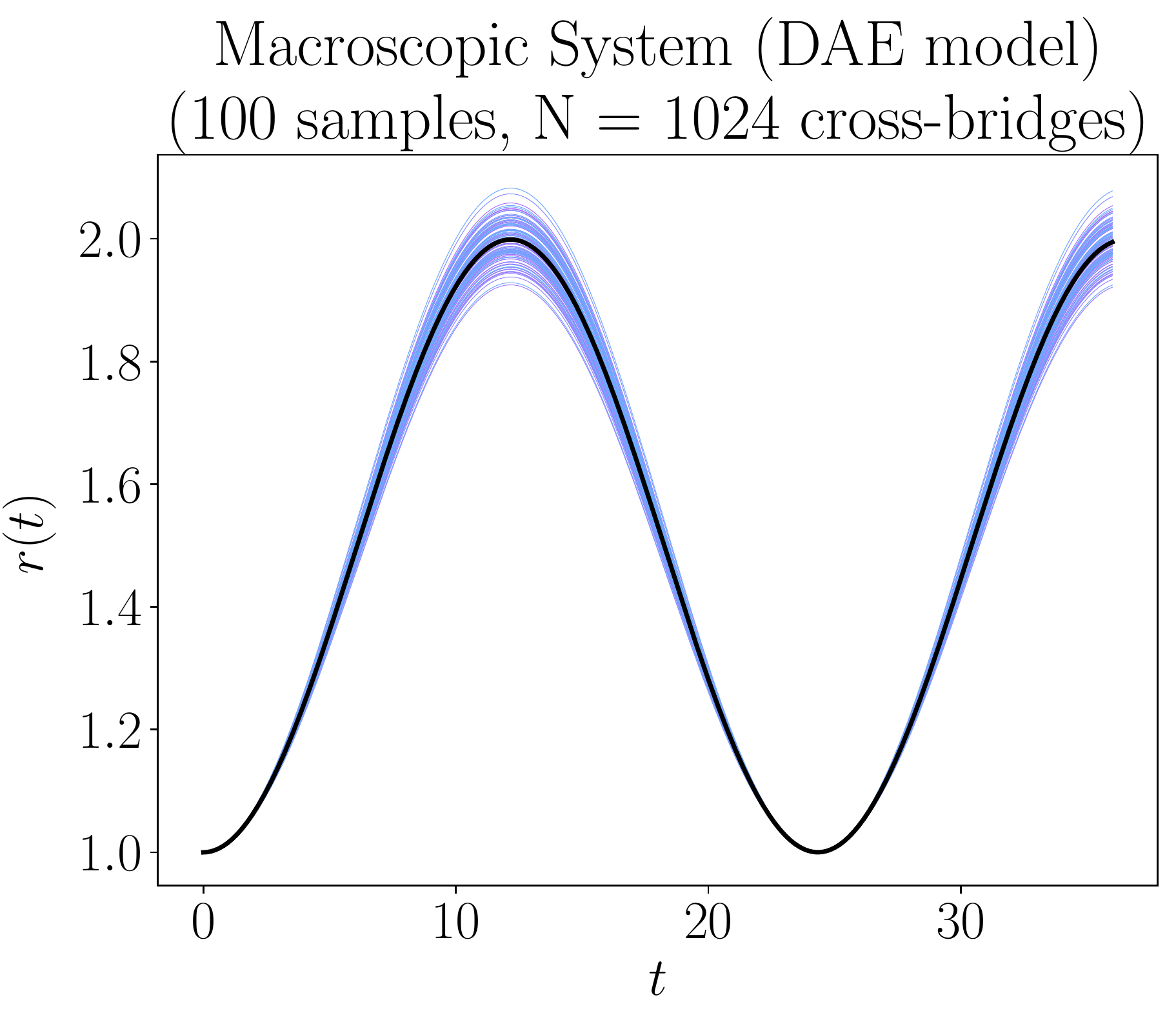}
	\hfill
	\includegraphics[width=0.45\textwidth]{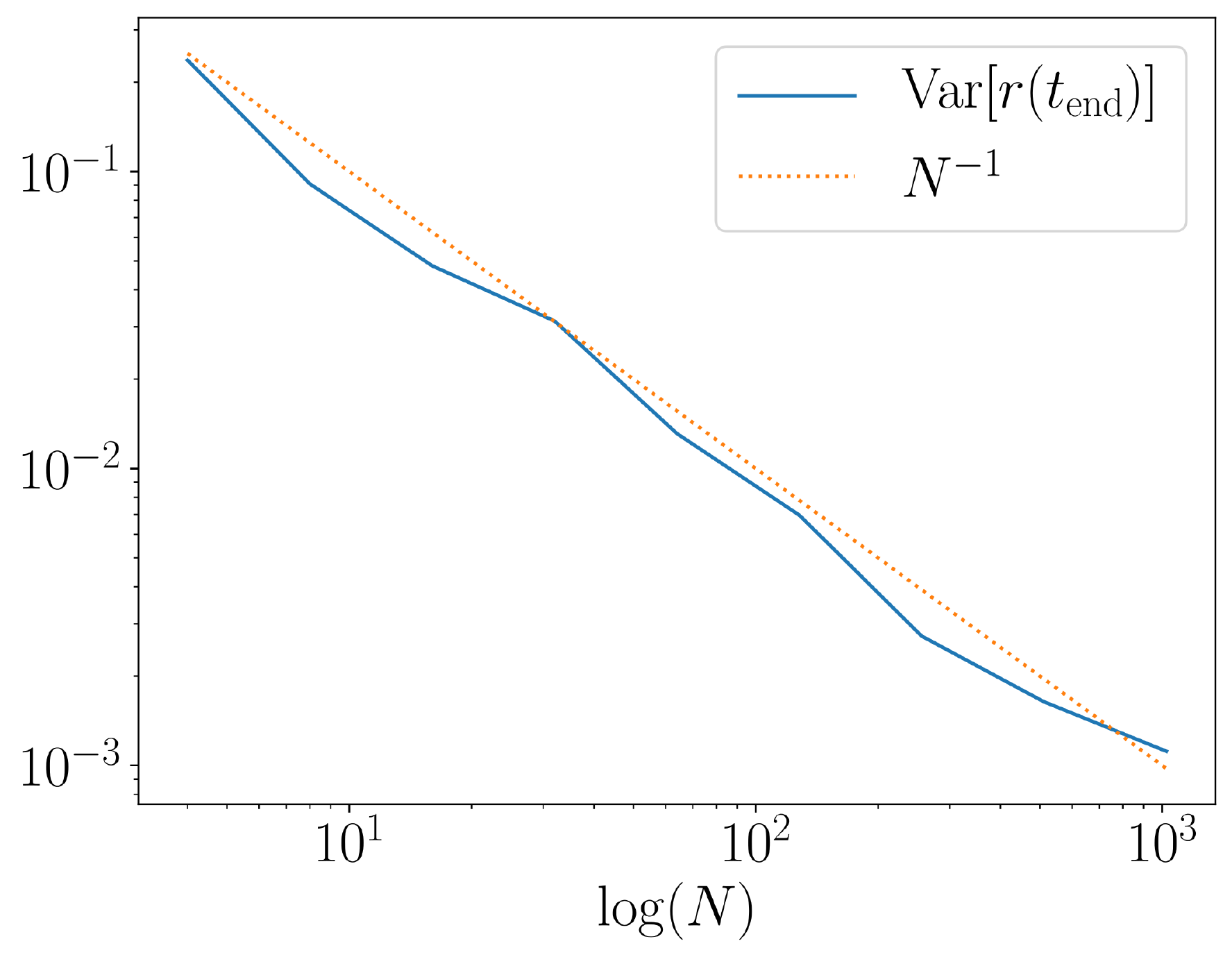}
	\caption{Samples of trajectories $r(t)$ of the macroscopic system for different numbers of cross-bridges $\N \in \{4,128,2048\}$ (top and bottom left). For an increasing number of particles $n$, the estimated variance of $r(t)$ decreases as $\frac{1}{n}$. (bottom right).}
	\label{fig:linear_mc}
\end{figure}

The mean-trajectories in \cref{fig:linear_mc} indeed converge towards the trajectory of the partially kinetic systems,
as visualised in \cref{fig:linear_mf_error}. We remark that this convergence also depends on the number of samples $n_{\text{samples}}$. Increasing the number of samples leads to faster convergence in \cref{fig:linear_mf_error}.
In a nonlinear setting, this convergence behaviour might change radically.

\begin{figure}[h]
	\centering
	\begin{minipage}{0.45\textwidth}
		\includegraphics[width=\textwidth]{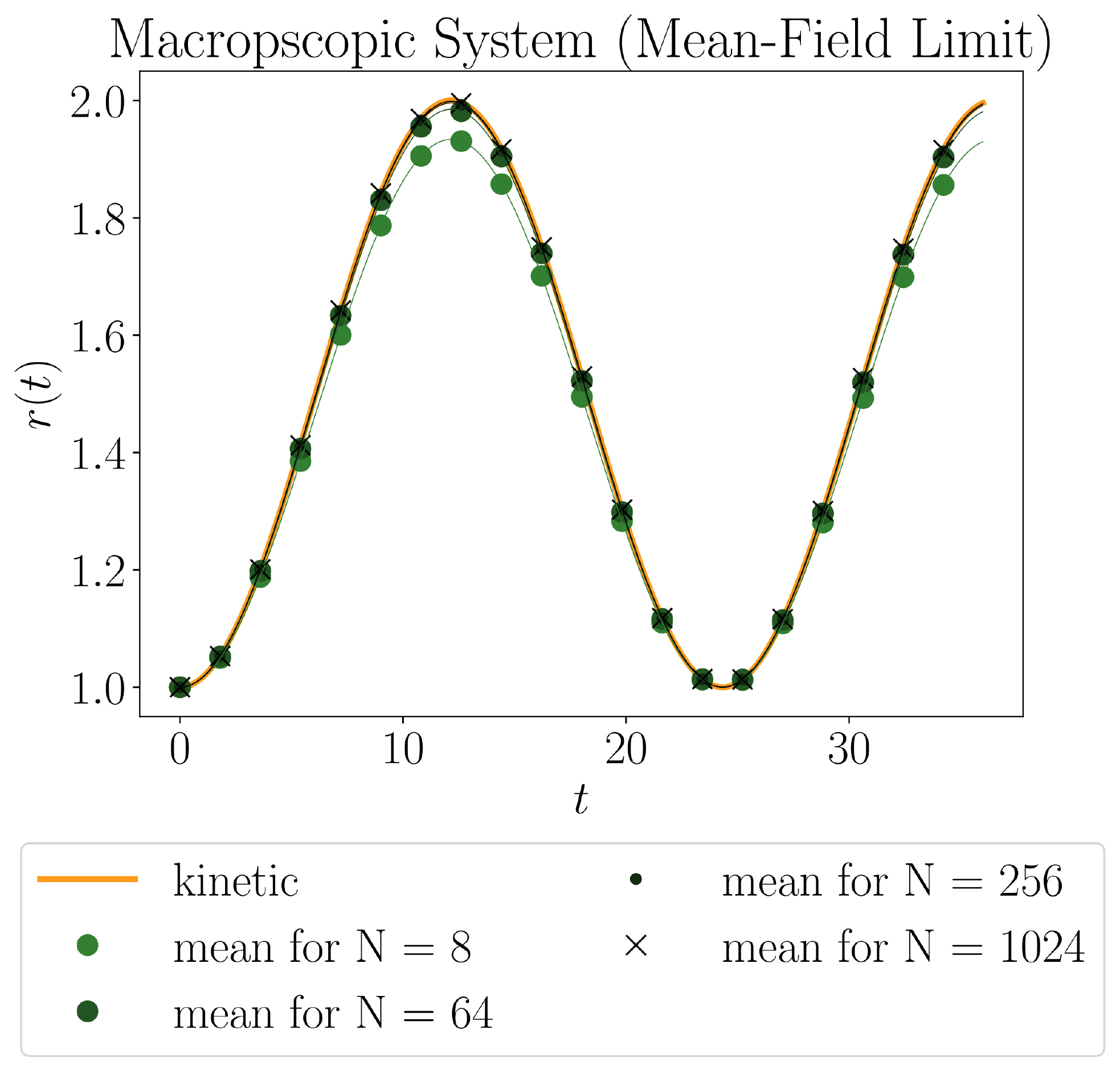}
	\end{minipage}
	\hfill
		\begin{minipage}{0.45\textwidth}
			\includegraphics[width=\textwidth]{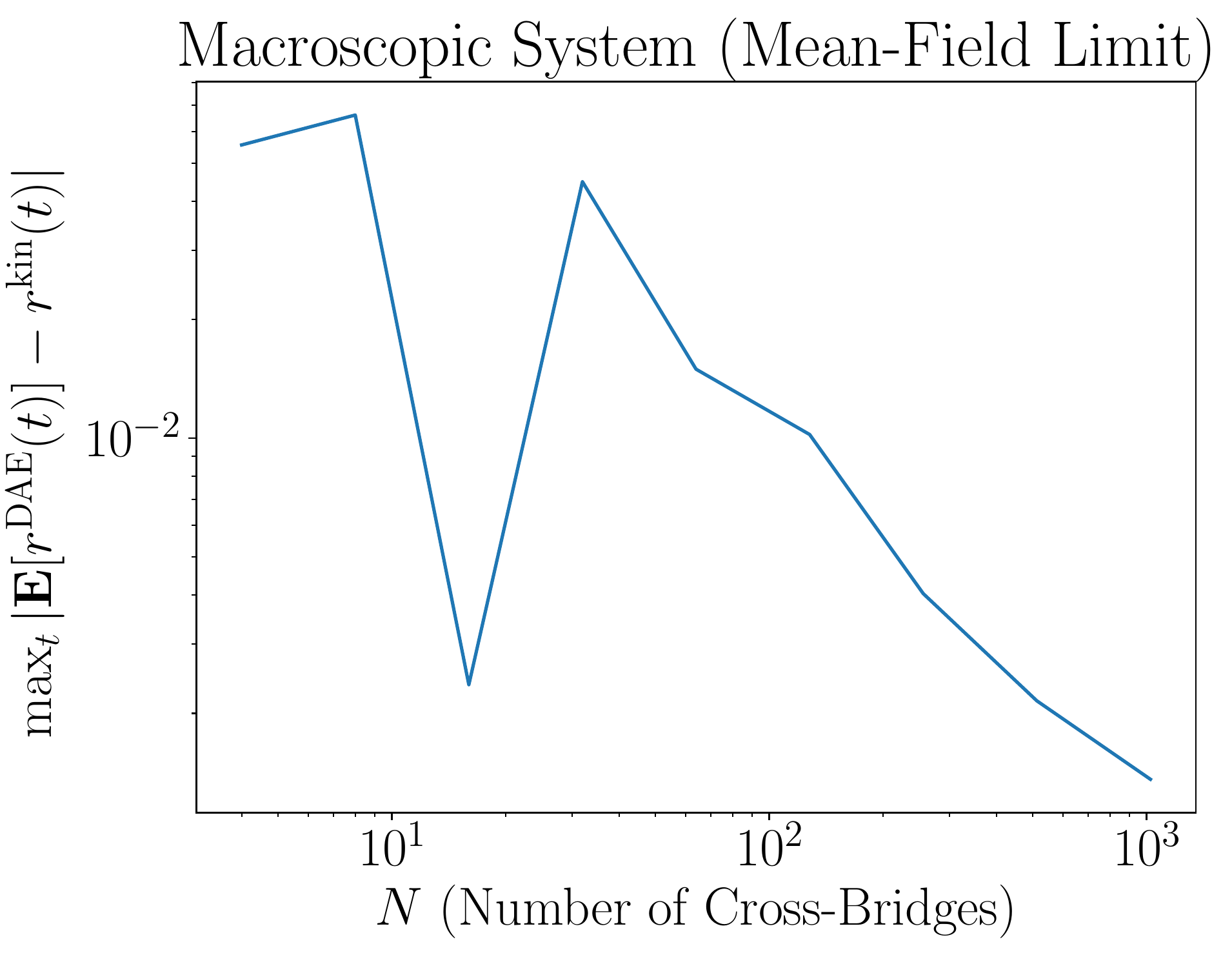}
		\end{minipage}
	\caption{The left plot compares the trajectory of the kinetic system and mean-trajectories of the discrete dynamics. This visual convergence is quantified in the right figure, which shows that the maximum distance between the mean-trajectories and the kinetic trajectory converges numerically.}
	\label{fig:linear_mf_error}
\end{figure}

\subsection{The Mean-Field Limit and Dobrushin's Stability Estimate for Linear Kinetic Systems}
\label{subsec:mean_field_dobr}

In this section, we will state an adapted version of Dobrushiun's stability estimate and define mean-field convergence for linear partially kinetic systems.
The proof of the stability estimate is given in \cref{subsec:proof_dobr}. The classical version of Dobrushin's stability estimate for the Vlasov equation can be found in \cite[theorem 1.4.3]{golseDynamicsLargeParticle2016}.

The numerical experiments in \cref{subsec:mean_field_numerical} indicate that the mean-trajectories of the macroscopic systems converge to the mean-field trajectory.
Now, to compare the distance between the discrete cross-bridge extensions $(\qd_1(t),\dots,\qd_\N(t))$ and the kinetic cross-bridge distribution $\pmst$, we need to choose a metric.
The empirical measure $\empn_{\qd_1(t),\dots,\qd_N(t)} \in \PRq{1}$ is a possibility to represent the discrete cross-bridge extensions in the space of probability measures, which leads to the new goal of finding a metric for the space $\PRq{1}$.
Moreover, such a metric should be compatible with the discrete distance between cross-bridge states. Therefore, we require
\begin{align}
\mathrm{dist}\left(\empn_{\qd_1,\dots,\qd_N}, \empn_{\tilde \qd_1,\dots,\tilde \qd_N} \right) \approx \frac{1}{\N} \sum_{i=1}^\N \norm{\qd_i - \tilde \qd_i}.
\label{eq:geom_motivation_W1_dist}
\end{align}
If $\qd_i \approx \tilde \qd_i$ holds, then the Monge-Kantorovich distance (also called Wasserstein distance) \index{Monge-Kantorovich distance}\index{Wasserstein metric|see {Monge-Kantorovich distance}} satisfies the geometric requirement \cref{eq:geom_motivation_W1_dist}.
For a detailed study of the Monge-Kantorovich distance, we refer to \cite[Chapter 6]{villaniOptimalTransportOld2008}.
For our purpose, the duality formula for the Monge-Kantorovich distance with exponent $1$ is most useful and hence serves as the definition here.

\begin{definition}[Monge-Kantorovich distance {\cite[Proposition 1.4.2]{golseDynamicsLargeParticle2016}}]
\index{Monge-Kantorovich distance}
\label{def:W1_dual}~\\
The Monge-Kantorovich distance with exponent $1$ is given by the formula
\begin{align}
W_1(\nu,\mu) =
\sup_{ \stackrel{\phi \in \LipRq}{\Lip(\phi) \leq 1} }
\abs{ \int_\Rq \phi(\qc) \dq \nu - \int_\Rq \phi(\qc) \dq \mu }, 
\label{eq:W1_dual}
\end{align}
with the notation 
\[
\Lip(\phi) \coloneqq \sup_{\stackrel{x, y \in \Rq}{x \neq y}} \frac{\abs{\phi(x) - \phi(y)}}{\norm{x-y}}
\]
and
\[
\LipRq \coloneqq \{ \phi : \Rq \to \RR \mid \Lip(\phi) < \infty \}.
\]
\index{Lipschitz continuous}
\end{definition}
The Monge-Kantorovich distance is a complete metric on $\PRq{1}$ \cite[Lemma 6.14]{villaniOptimalTransportOld2008}.

For the attached cross-bridge model with $\nq = 1$, the duality formula \cref{eq:W1_dual} with $\phi(\qc) = \qc$ is directly applicable to estimate the difference of the mean-field forces
\begin{align}
\norm{\fmf(\nu) - \fmf(\mu)}
&\leq \norm{\cDphi^T \kappaq} \norm{ \int_\Rq \qc \dq \nu - \int_\Rq \qc \dq \mu} \\
&\leq  \norm{\cDphi^T \kappaq} W_1(\nu,\mu).
\end{align}
This estimate is at the core of the relation between partially kinetic systems and the Monge-Kantorovich distance. \index{Stability}
Together with a classical stability estimate for ODEs, this yields the following theorem.

\begin{theorem}[Dobrushin's Stability Estimate for Linear Partially Kinetic Systems]
\label{thm:linear_dobrushin}~\\
\index{Stability!Dobrushin's estimate}\index{Dobrushin's stability estimate}
Suppose that for $i \in \{1,2\}$ the tuples $(r_i(t),\pms_i(t)) \in \Rr \times \PRq{1}$ are solutions of  
\cref{eq:muscle_newton_eff,eq:muscle_flow,eq:muscle_pushforward}
with initial conditions
\[
r_i(0) = \ri_i, \quad \dot r_i(0) = \si_i \quad \text{and} \quad 
\pms_i(0) = \pmsin_i.
\]
Then
\begin{align}
\norm{r_1(t) - r_2(t)} + \norm{\dot r_1(t) - \dot r_2(t)}
+ &W_1(\pms_1(t), \pms_2(t)) 
\notag \\
\MoveEqLeft \leq C e^{Lt} \left(
\norm{\ri_1 - \ri_2} + \norm{\si_1 - \si_2} 
+ W_1(\pmsin_1, \pmsin_2) \right),
\label{eq:dobr_stab_estm}
\end{align}
for some constants $L, C > 0$ which are independent of the initial conditions.
\end{theorem}

The proof of \cref{thm:linear_dobrushin} is content of \cref{subsec:proof_dobr}.
\color{black}

Dobrushin's stability estimate for linear partially kinetic systems \cref{eq:dobr_stab_estm} provides a concrete answer to the approximation quality of the kinetic description.
\cref{lem:insert_emperical} shows that solutions $(r(t),\qd_{1}(t),\dots,\qd_{N}(t))$ of the discrete system
\cref{eq:dm_newton_eff,eq:dm_constr_ind1} yield a solution of the kinetic formulation with initial data $\pmsin = \empn_{\qdi_{1},\dots,\qdi_{N}}$.
If we increase the number of cross-bridges such that the empirical measures \index{Empirical measure} converge \index{Convergence!mean-field}\index{Mean-field limit} to a probability distribution $\pmsin \in \PRq{1}$, i.e.
\[
W_1(\empn_{\qdi_1,\dots,\qdi_N}, \pmsin) \to 0, \quad \text{for~} N \to \infty,
\]
then \cref{eq:dobr_stab_estm} provides a bound for the approximation error \index{Stability}
\begin{align}
\norm{r_N(t) - r(t)}
+
\norm{s_N(t) - s(t)}
+W_1(\empn_{Q_1(t),\dots,Q_N(t)}, \pmst) 
&\leq C e^{tL} 
W_1(\empn_{\qdi_1,\dots,\qdi_N}, \pmsin) \notag\\
& \to 0, \quad \text{for~} N \to \infty.
\label{eq:mean_field_limit}
\end{align}
This estimate provides a rigorous argument for the use of kinetic models.
Moreover, if the initial distribution of $N$ cross-bridges $(\qd_j)_{j=1,\dots,N}$ is very close to  a continuous distribution $\pmsin$, then approximation error of the kinetic description is bounded.

Using tools from probability theory and functional analysis, the topology for the convergence in \cref{eq:mean_field_limit} can be refined to the topology of weak convergence \index{Convergence!weak} of measures (also called convergence in distribution for probability measures) \cite[theorem 6.9]{villaniOptimalTransportOld2008}, \cite[Lemma 1.4.6]{golseDynamicsLargeParticle2016}. This leads to the precise definition of mean-field convergence \cite{jabinReviewMeanField2014}. We omit these details here as they are out of scope for the present paper.

\subsection{Proof of Dobrushin's Stability Estimate for Linear Partially Kinetic Systems}
\label{subsec:proof_dobr}

This section provides a proof of \cref{thm:linear_dobrushin}. Compared to classical proofs of Dobrushin's stability estimate for the Vlasov equation \cite{golseDynamicsLargeParticle2016,jabinReviewMeanField2014}, the particular structure of linear partially kinetic system allows for a more elementary proof, which essentially reuses the stability estimate for ODEs with respect to initial conditions and parameters, as stated in \cref{lem:fundamental_ode}.

In contrast to many other metrics for probability measures, the Monge-Kantorovich distance 
is motivated by geometry. In particular, shifting two measures increases their distance at most by the length of the shift, see \cref{lem:W1_shift}. Due to this property, the proof of \cref{thm:linear_dobrushin} boils down to the ODE estimate from \cref{lem:fundamental_ode}.

\begin{lemma}[Shift of the Monge-Kantorovich distance]
\index{Invariance!shift}
\label{lem:W1_shift}
Let $\mu, \nu \in \PRq{1}$ be two probability measures with finite first moment.
For $w \in \Rq$ we define the shift mapping as $T_w: \Rq \to \Rq : \qc \mapsto \qc + w$.
Then for $w_1,w_2 \in \Rq$ it holds
\[
W_1(\push{T_{w_1}} \mu, \push{T_{w_2}}{\nu}) \leq W_1(\mu,\nu) + \norm{w_2 - w_1}.
\]
\end{lemma}

\cref{lem:W1_shift} is usually a special case of more general theorems on the relation between geodesic flows and the Monge-Kantorovich metric \cite[Chapter 8]{villaniOptimalTransportOld2008}. For completeness, we give an elementary proof.

\begin{proof}
The Monge-Kantorovich distance is invariant with respect to shifts
$W_1(\push{T_w}{\nu}, \push{T_w}{\mu}) = W_1(\nu, \mu)$, hence we can assume $w_1 = 0$.
For an arbitrary vector $w \in \Rq$, we use the duality formula \cref{eq:W1_dual} and compute
\begin{align}
W_1(\nu, \push{T_w}{\mu}) 
&=
\sup_{\Lip(\phi)\leq1}
\abs{ \int_\Rq \phi(\qc) \dq \nu - \int_\Rq \phi(\qc) \dq {(\push{T_w}{\mu})} } \notag \\
&=
\sup_{\Lip(\phi)\leq1}
\abs{ \int_\Rq \phi(\qc) \dq \nu - \int_\Rq \phi(T_w(\qc)) \dq \mu} \notag  \\
&=
\sup_{\Lip(\phi)\leq1}
\abs{ \int_\Rq \phi(\qc) \dq \nu - \int_\Rq \phi(\qc) \dq \mu + \int_\Rq \phi(\qc) - \phi(\qc + w) \dq \mu} \notag  \\
&\leq W_1(\nu,\mu) + 
\sup_{\Lip(\phi)\leq1}
\abs{\int_\Rq \phi(\qc) - \phi(\qc + w) \dq \mu}
\end{align}
where we have applied the transformation formula to the integral  of the pushforward operator.
Next, we employ $\Lip(\phi) \leq 1$ to obtain the upper bound
\begin{align*}
\abs{\int_\Rq \phi(\qc) - \phi(\qc + w) \dq \mu} 
&\leq 
\int_\Rq \abs{\phi(\qc) - \phi(\qc + w)} \dq \mu \\
&\leq 
\int_\Rq \norm{w} \dq \mu = \norm{w}.
\end{align*}
We conclude 
\begin{align}
W_1(\nu, \push{T_w}{\mu}) \leq W_1(\nu, \mu) + \norm{w}.
\label{eq:W1_shift_ineq}
\end{align}
\qed
\end{proof}

\begin{theorem}[{The \textquote{Fundamental Lemma}, {\cite[theorem 10.2]{hairerSolvingOrdinaryDifferential2009}}}]
\index{Stability!ODE}
\label{lem:fundamental_ode}
Suppose that for $i \in \{1,2\}$ the functions $x_i(t)$ solve
\begin{align}
\dot x_i(t) &= f_i(x_i(t)),\\
 x_i(0) &= \xin_i.
\end{align}
If for some constants $\varrho, \varepsilon, L > 0$ the following bounds hold
\begin{enumerate}[label=\roman*)]
\item $\norm{\xin_1 - \xin_2} \leq \varrho$,
\item $\norm{f_1(x) - f_2(x)} \leq \varepsilon \quad \text{for all~} x \in \RR^n$,
\item $\norm{f_1(a) - f_1(b)} \leq L \norm{a-b} \quad \text{for all~} a,b \in \RR^n$,
\end{enumerate}
then
\begin{align}
\norm{x_1(t) - x_2(t)} \leq \varrho e^{L t} + \frac{\varepsilon}{L} \left( e^{Lt} - 1 \right).
\end{align}
\end{theorem}

\cref{thm:linear_dobrushin} can be considered as a generalisation of \cref{lem:fundamental_ode}.

\begin{proof}[of \cref{thm:linear_dobrushin}]
First, we will reformulate the ODE formulation \cref{eq:muscle_newton_eff,eq:muscle_flow,eq:muscle_pushforward} such that \cref{lem:fundamental_ode} is applicable.
The characteristic flow is explicitly given by
\begin{align}
\cflow(t,\qi) \coloneqq -\cDphi (r(t) - \ri) + \qi,
\label{eq:cflow_explicit}
\end{align}
which is the unique solution of \cref{eq:muscle_flow} and the initial condition.
With \cref{eq:cflow_explicit}, we can compute the first moment of $\pmst$ as
\begin{align}
\int_\Rq \qc \dq \pmst
 &= \int_\Rq \cflow(t,\qcp) \dqp \pmsin  \notag \\
&=  \cDphi (r(t) - \ri) + \int_\Rq \qcp \dqp \pmsin.
\label{eq:dobr_first_moment}
\end{align}
As a result, the effective force can be written as
\begin{align}
\feff(r(t),\pmst) 
&= -\kappar r(t) + \Nreal \cDphi^T \kappaq  \int_\Rq \qc \dq \pmst \notag\\
&= -\kappar r(t) + \Nreal \cDphi^T \kappaq \left( \cDphi (r(t) - \ri) + \int_\Rq \qc \dq \pmsin \right) \notag\\
&\eqqcolon \tilde f_{\mathrm{eff}}(r(t); \pmsin). \notag
\end{align}

Now, we consider two different initial conditions $(\ri_1,\si_1,\pmsin_1) \in \Rr \times \Rr \times \PRq{1}$ and $(\ri_2,\si_2,\pmsin_2) \in \Rr \times \Rr \times \PRq{1}$.
To prepare the application of \cref{lem:fundamental_ode}, we transform $\meff \ddot r_i = \feff$ \cref{eq:muscle_newton_eff} into a first order ODE,
with $x_i(t) = ( r_i(t), \dot r_i(t) )^T \in \RR^{2 \nr}$ and
\begin{align}
\dot x_i(t) &= \begin{pmatrix}
\dot r_i(t) \\
\meff^{-1} \tilde f_{\mathrm{eff}}(r(t); \pmsin_i)
\end{pmatrix}
\eqqcolon f_i(x_i(t)) \quad &\text{for~} i \in \{1,2\},\\
x_i(0) &= (\ri_i,\si_i)^T  \quad &\text{for~} i \in \{1,2\}. \notag
\end{align} 
We remark that $\meff = \Mr + \Nreal \cDphi^T \Mq \cDphi$ is invertible, since $\Mr$ and $\Mq$ are defined to be positive definite, see \cref{sec:dae_model}.

The difference between the right-hand sides $f_i$ at a fixed state $x = (r,s) \in \RR^{2 \nr}$ is
\begin{align}
\norm{f_1(r) - f_2(r)} \notag 
&= \norm{ \meff^{-1} \left(  \tilde f_{\mathrm{eff}}(r; \pmsin_1) -  \tilde f_{\mathrm{eff}}(r; \pmsin_2) \right) } \notag \\
&=\norm{\meff^{-1} \Nreal \cDphi^T \kappaq \left( \cDphi (\ri_1 - \ri_2) + \int_\Rq \qc \dq {\pmsin_1} - \int_\Rq \qc \dq {\pmsin_2} \right)} \notag \\
&\eqqcolon \varepsilon.
\label{eq:dobr_vareps}
\end{align}
Next, we compute the Lipschitz constant of $f_1$.
The partial derivatives are 
\[
\dpd{f_1}{\dot r_1} = 1 \quad \text{and} \quad \dpd{f_1}{r_1} = \meff^{-1} ( \kappar + \Nreal \cDphi^T \kappaq \cDphi ),
\]
which implies that 
\begin{align*}
L &\coloneqq \norm{\meff^{-1}} \left( \norm{\kappar} +  \Nreal \norm{\cDphi^T \kappaq \cDphi} \right) + 1
\end{align*}
is a Lipschitz constant for $f_1$.

Then \cref{lem:fundamental_ode} yields
\begin{align}
\sqrt{ \norm{r_1(t) - r_2(t)}^2 + \norm{s_1(t) - s_2(t)}^2}
&\leq
\varrho e^{L t}
+ \frac{\varepsilon}{L} \left( e^{L t} - 1 \right)
\label{eq:dobr_fund_estm}
\end{align}
with
\begin{align}
\varrho &\coloneqq \norm{x_1(0) - x_2(0)} = 
\sqrt{ \norm{\ri_1 - \ri_2}^2 + \norm{\si_1 - \si_2}^2}.
\end{align}

Next, we apply the duality formula for the Monge-Kantorovich distance \cref{eq:W1_dual} to each component of $\qc \in \Rq$. Hence, we use $\phi(\qc) \coloneqq \qc_l \in \RR$ in \cref{eq:W1_dual}, which yields
\begin{align}
\norm{ \int_\Rq \qc \dq {\pmsin_1} - \int_\Rq \qc \dq {\pmsin_2}} \notag \\
\MoveEqLeft[6]\leq 
\sum_{l=1}^\nq \abs{ \int_\Rq \qc_l \dq {\pmsin_1} - \int_\Rq \qc_l \dq {\pmsin_2}}
\notag \\
\MoveEqLeft[6]\leq 
\nq W_1(\pmsin_1,\pmsin_2). \label{eq:W1_for_nq}
\end{align}
The estimate \cref{eq:W1_for_nq} provides the upper bound
\begin{align}
\varepsilon 
&\leq \norm{\meff^{-1} \Nreal \cDphi^T \kappaq} \left( \norm{\cDphi} \norm{ \ri_1 - \ri_2 } + \norm{\int_\Rq \qc \dq {\pmsin_1} - \int_\Rq \qc \dq {\pmsin_2}} \right) 
\notag \\
&\leq C_1 ( \varrho + W_1(\pmsin_1,\pmsin_2) )
\notag 
\end{align}
with $C_1 \coloneqq \enVert[0]{\meff^{-1} \Nreal \cDphi^T \kappaq} (\norm{\cDphi} + \nq)$.

Now, we use $\frac{1}{2}(\abs{a}+\abs{b}) \leq \sqrt{a^2 + b^2}$
to transform \cref{eq:dobr_fund_estm} into
\begin{align}
\norm{r_1(t) - r_2(t)} + \norm{s_1(t) - s_2(t)} 
&\leq 
2 \sqrt{ \norm{r_1(t) - r_2(t)}^2 + \norm{s_1(t) - s_2(t)}^2} \notag \\
&\leq
2 \rho e^{Lt} + 2 \frac{C_1 (\rho + W_1(\pmsin_1,\pmsin_2))}{L}  ( e^{Lt} - 1) \notag\\
&\leq
2 (1 + \frac{C_1}{L}) \rho e^{Lt} + 2 \frac{C_1}{L}  W_1(\pmsin_1,\pmsin_2)) e^{Lt} \notag
\intertext{and we bound $\varrho$ with $\sqrt{a^2 + b^2} \leq \abs{a} + \abs{b}$ to obtain}
\norm{r_1(t) - r_2(t)} + \norm{s_1(t) - s_2(t)} \notag \\ 
\MoveEqLeft[5] \leq  C_2 \left( \norm{\ri_1-\ri_2} + \norm{\si_1 - \si_2} + W_1(\pmsin_1,\pmsin_2) \right) e^{Lt}
\label{eq:dobr_heavy_estm}
\end{align}
with the constant $C_2 \coloneqq 2(1+\frac{C_1}{L})$.

The estimate \cref{eq:dobr_heavy_estm} already looks similar to the claim. It only misses an estimate for the difference of the cross-bridge distributions $\pms_1(t)$ and $\pms_2(t)$. 
We notice that 
\[
\pms_i(t) = \push{\cflow_i(t,\cdot)}{\pmsin_i} = T_{w_i} \pmsin_i
\]
with 
\[
w_i(t) = -\cDphi ( r_i(t) - \ri_i ).
\]
Therefore, \cref{lem:W1_shift} gives
\begin{align}
W_1(\pms_1(t), \pms_2(t)) 
&=
W_1(T_{w_1(t)} \pmsin_1, T_{w_2(t)} \pmsin_2) \notag
\\
\MoveEqLeft[5] =
W_1(\pms_1(t), T_{w_2(t) - w_1(t)} \pms_2(t)) \notag
\\
\MoveEqLeft[5] \leq
W_1(\pmsin_1,\pmsin_2) + \norm{w_2(t) - w_1(t)} \notag 
\\
\MoveEqLeft[5] \leq
W_1(\pmsin_1,\pmsin_2) + \norm{\cDphi} \left( \norm{r_1(t) - r_2(t)} + \norm{\ri_1 - \ri_2} \right).
\notag \\
\MoveEqLeft[5] \leq C_2 (\norm{G_r} + 1) \left( \norm{\ri_1-\ri_2} + \norm{\si_1 - \si_2} + W_1(\pmsin_1,\pmsin_2) \right) e^{Lt}
\label{eq:dobr_particles_estm}
\end{align}
where we use $e^{Lt} \geq 1$ and \cref{eq:dobr_heavy_estm} in the last equation.

Combining \cref{eq:dobr_heavy_estm,eq:dobr_particles_estm}, we obtain \cref{eq:dobr_stab_estm} with the constants
\begin{align*}
C \coloneqq C_2 ( 2 + \norm{G_r}), \quad
L \coloneqq \norm{\meff^{-1}} \left( \norm{\kappar} +  \Nreal \norm{\cDphi^T \kappaq \cDphi} \right) + 1.
\end{align*}
\qed
\end{proof}

\section{Generalisations and Relations to Established Models}
\label{sec:generalisations}

The model for attached cross-bridges demands for extensions in two directions:
For applications in biology, the cross-bridge model should incorporate more biological effects, for example, cross-bridge cycling.
For further mathematical investigations, partially kinetic systems can be studied in more generality, most notably with nonlinear constraints or stochastic jumps.
In this section, we give a brief outlook on these extensions.

\subsection{An Abstract Class of Nonlinear Partially Kinetic Systems}
\label{subsec:abstract_pks}

The underlying system \cref{eq:dm_newton_r,eq:dm_newton_qj,eq:dm_constr_ind3} is a prototype for linear partially kinetic systems. A possible extension would be to consider systems with nonlinear forces $\Fr(r), \Fq(\qd_j)$ and with uniform but nonlinear constraints $g(r,\qd_j) = g(\ri,\qdi_j)$, for example
\begin{align}
\Mr \ddot r &= \Fr(r) - \sum_{i=1}^{\N} \dpd{g^T}{r} \lambda_i , 
\label{eq:nonlin_newton_r}
\\
\Mq \ddot \qd_j &= \Fq(\qd_j) - \dpd{g^T}{\qd_j}\lambda_j \quad \text{for~} j=1,\dots,\N, 
\label{eq:nonlin_newton_qj}
\\
g(r,\qd_j)  &= g(\ri,\qdi_j)  \quad \text{for~} j=1,\dots,\N.
\label{eq:nonlin_constr_ind3}
\end{align}
Despite the nonlinear terms, the formal derivation of kinetic equations for \cref{eq:nonlin_newton_r,eq:nonlin_newton_qj,eq:nonlin_constr_ind3} follows similar algebraic steps as in \cref{subsec:kin_char_flow_ode,subsec:kin_char_flow_dae}.

\begin{remark}
\label{rmk:nonlinear_dobr}
We conjecture that \cref{thm:linear_dobrushin} generalises to nonlinear partially kinetic systems \cref{eq:nonlin_newton_r,eq:nonlin_newton_qj,eq:nonlin_constr_ind3}, if certain Lipschitz conditions and technical bounds are satisfied.
In contrast to the linear case, the effective mass matrix $\meff$ depends on the cross-bridge distribution $\pmst$.
Therefore, it is more challenging to obtain the Lipschitz constant $L$ and the 
defect $\varepsilon$ as in \cref{eq:dobr_vareps}.
Moreover, the mean-field characteristic flow $\cflow(t,\cdot)$ will be nonlinear and there is no explicit formula as in \cref{eq:cflow_explicit}. Therefore \cref{lem:W1_shift} is not applicable in the nonlinear case.
\end{remark}

\subsection{Coupling with Nonlinear Elasticity}
\index{Elasticity}
The muscle model in \cref{sec:dae_model}  allows only linear constraints and linear forces. This is not sufficient for realistic multi-scale models.
On a large scale, muscles can be modelled as nonlinear, quasi-incompressible, hyperelastic solids \cite{simeonModelMacroscaleDeformation2009}. 

Using the framework of partially kinetic systems, a nonlinear constraint can link a hyperelastic model at the large scale for the muscle tissue and the cross-bridge model at the physiological scale. The resulting constraint is linear with respect to the extension of cross-bridges but nonlinear with respect to the deformation of the tissue. In this sense, the situation is similar to \cref{subsec:abstract_pks}
Additional complexity arises since the macroscopic system in this setting is described by the PDEs of elasticity, which results in an infinite-dimensional system already in the discrete case.
We can assume that most material points of the muscle are occupied parallel actin and myosin filaments. This assumption leads mathematically to an infinite family of cross-bridge models, one at each spatial point of the muscle. A formal derivation is possible and not very different from the theory presented in this article. Analytical results, however, are far more challenging in this setting.
Even more, the cross-bridges of sarcomeres at neighbouring spacial points can be in very different states. Therefore, a rigorous mathematical approach requires a justification for spatial averaging over the cross-bridge states.

\subsection{Comparison with Established Cross-Bridge Models}
\label{subsec:compare_estab_models}

The model for attached cross-bridges neglects a fundamental part of cross-bridge dynamics:
Cross-bridges can attach and detach dynamically. \index{Sliding filament theory} The repeated attachment and detachment is called cross-bridge cycling.\index{Cross-bridges!cycling} Only with this mechanism, muscle cells can contract far beyond the working range of a single cross-bridge. 
Non-kinetic models, like the popular Hill model fail to 
capture some phenomena which depend on cross-bridge cycling \cite[Section 15.3.1]{keenerMathematicalPhysiology2009}. This motivates the use of kinetic models in muscle simulations.

The most common models for cross-bridge cycling are probabilistic. One example is the \emph{two-state model} \cite{zahalakDistributionmomentApproximationKinetic1981}.
The two-state model \index{Cross-bridges!two-state model} is usually only formulated on the kinetic level as a source term in the transport equation \cref{eq:muscle_transport_eq_first}. A function $h_+(t,\qc,u)$ determines the creation rate of new cross-bridges with extension $\qc \in \Rq$ at time $t$. The counterpart is a function $h_-(t,\qc,u)$, which gives the annihilation rate of cross-bridges with extension $\qc \in \Rq$.
The two-state model leads to a kinetic transport equation with source terms
\begin{align}
\dpd{\pdens}{t}(t,\qc) + \veff \dpd{\pdens}{\qc}(t,\qc) = h_+(t,\qc,u(t,\qc)) - h_-(t,\qc,u(t,\qc)).
\label{eq:muscle_transport_with_sources}
\end{align}
\index{PDE!transport}
Here, the transport velocity $\veff$ is the contraction speed of the muscle. In the setting of \cref{subsec:muscle_mf_pde}, contraction speed of the muscle is the velocity of the macroscopic system, hence $\veff = - \cDphi \dot r(t)$.
The rate functions $h_\pm$ allow controlling the contraction speed in the muscle model.
For a contracting muscle, they are such that many cross-bridges with large positive extension are created, and cross-bridges with negative extension are annihilated.
A discussion of the two-state model is not subject of this publication. Instead, we refer to \cite{keenerMathematicalPhysiology2009,howardMechanicsMotorProteins2001,herzogSkeletalMuscleMechanics2000}.
Models with more than just two-states are also studied and applied, for example in \cite{heidlaufMultiscaleContinuumModel2016,herzogSkeletalMuscleMechanics2017,heidlaufContinuummechanicalSkeletalMuscle2017}.

All these models for cross-bridge cycling have in common, that their underlying discrete microscopic model for cross-bridge cycling is, to the best of our knowledge, not specified. Instead, these models are built directly from the kinetic perspective, to avoid unnecessary complexity. The rate functions $h_+$ and $h_-$ are heuristic and usually fitted to experimental data \cite{zahalakDistributionmomentApproximationKinetic1981}.
We are not aware of a derivation for the rates $h_+$ and $h_-$
using kinetic theory. As far as we know, the most rigorous approach to derive the rate functions is based on thermodynamics principles \cite{maDistributionmomentModelEnergetics1991}.

Nonetheless, a rigorous connection between the microscopic world and macroscopic simulations requires a microscopic law for cross-bridge cycling.
One possibility could be to model the microscopic law as a pieces-wise deterministic Markov process \index{Markov process}
\cite{davisPiecewiseDeterministicMarkovProcesses1984} where a Markov process models the creation and annihilation of cross-bridges, and the deterministic model for attached cross-bridges \cref{eq:dm_newton_eff,eq:dm_newton_qj,eq:dm_constr_ind3} governs the system at all other times.
The corresponding mean-field limit is formally similar to the mean-field limit for chemical reactions, as outlined in \cite[Section 3.3]{darlingDifferentialEquationApproximations2008}.

Finally, we want to point out a detail which differs between established models and the kinetic model developed in this article. Since we developed a  kinetic theory which includes the constraints from the beginning, the influence of the cross-bridges onto the macroscopic system is exactly represented.
The effective balance law which governs the macroscopic system is
\cref{eq:muscle_newton_eff}
\[
\meff \ddot r = \feff,
\]
where $\feff$ includes the influence of the cross-bridges onto the macroscopic system, i.e. the muscle tissue.
Moreover, $\meff$ integrates the inertia of the cross-bridges to the effective balance law \cref{eq:muscle_newton_eff}.
In contrast, established muscle models \cite{maDistributionmomentModelEnergetics1991,heidlaufContinuummechanicalSkeletalMuscle2017,bolMicromechanicalModellingSkeletal2008} just compute the force $\feff$ and use 
\[
\Mr \ddot r = \feff.
\]
For applications, this approximation is very reasonable, since the momentum of cross-bridges $\Mq$ is considered to be very small, compared to the mass of the muscle tissue, i.e. the underlying assumption is 
$\Nreal \norm{\Mq} \ll \norm{\Mr}$, which implies $\Mr \approx \meff$.
\color{black}

\section{Numerical Simulation of Partially Kinetic Systems}
\label{sec:numerical_examples}

Partially kinetic systems are mixed systems involving Newton's equations of motion for the macroscopic components \cref{eq:muscle_newton_eff} and a non-linear transport equation for the particle density \cref{eq:muscle_transport_eq_first}.
Therefore, a perfect numerical scheme for such a system should not only nearly conserve energy and momentum, but also the mass of the particle density. It is an open issue if such a scheme exists.

In the literature on sliding filament theory, a popular method is the distributed moment method (DM method) \index{Distributed moment method} \cite{zahalakDistributionmomentApproximationKinetic1981}.
By assuming a specific shape for the cross-bridge distribution $\pmst$, it is possible to derive a closed set of differential equations for the first moments of $\pmst$ and thus approximate the solution of the transport equation for cross-bridges \cref{eq:muscle_transport_with_sources} by a three-dimensional ODE. If the particle measure is close to a Gaussian distribution, then the DM method works best.
The DM method is successful and has been used in many multi-scale simulations \cite{bolMicromechanicalModellingSkeletal2008,heidlaufContinuummechanicalSkeletalMuscle2017}.
However, the state of cross-bridges is often very different from a normal distribution. In these cases,
the DM method does not yield a numerically convergent discretisation of \cref{eq:muscle_newton_eff,eq:muscle_transport_eq_first}.
This drawback is acceptable for most applications, but precise information about the physical state of the cross-bridges is lost.

\subsection{Implementation Details}

The numerical simulations in this article are performed straightforwardly.
More advanced and adapted methods are out of scope for this article.

For numerical time integration, we used the \texttt{RADAU} \cite[Section  IV.8.]{hairerStiffDifferentialalgebraicProblems2010} and the \texttt{LSODA} \cite{petzoldAutomaticSelectionMethods1983} methods from the python package \texttt{scipy.integrate} \cite{waltNumPyArrayStructure2011}, both with numerical parameters $\mathrm{atol} = 10^{-8}, \mathrm{rtol} = 10^{-8}$ for the adaptive time-stepping scheme. For all examples in this article, the time integration was successful without any indicators for numerical instabilities.

For space discretisation of the transport equation \cref{eq:muscle_transport_eq_first}, the standard upwind discretisation \index{Discretisation!upwind} \index{PDE!transport} was used \cite[Section 10.4]{levequeFiniteDifferenceMethods2007}.
For simplicity, we assumed zero boundary conditions for the numerical spatial domain. 
The grid was chosen sufficiently large such that the boundary conditions do not influence the simulation results.
In detail, we solved the transport equation \cref{eq:muscle_transport_eq_first} restricted onto the spacial domain $[-5,7]$ on an equidistant grid with $101$ grid-points,
with
\[
\dot u(t,y_i) = 
\begin{cases}
\veff(t) \frac{1}{\Delta x}( u(t,x_i) - u(t,x_{i-1}) )  & \veff(t) \geq 0, \\
\veff(t) \frac{1}{\Delta x}( u(t,x_{i+1}) - u(t,x_i) )  & \veff(t) < 0.
\end{cases}
\]
where $\veff(t) = - \cDphi \dot r(t)$ denotes the velocity according to \cref{eq:muscle_transport_eq_first},
and $\Delta x = \frac{12}{100}$ is the space between two grid-points $y_i = -5 + i \Delta x$.

For all simulations in this article, we used the model parameters from \cref{tab:params}.
\color{black}
The parameters are not related to real cross-bridges; the numerical simulations should merely demonstrate the mathematical model, not its application to reality.

\begin{table}
\centering
\begin{tabular}{lrr} 
\toprule
Model Parameters\\  
\midrule 
Description & Symbol & Value\\ 
\midrule 
degrees of freedom (macroscopic system) & $\nr$ & 1 \\
degrees of freedom (single cross-bridge) & $\nq$ & 1 \\
number of cross-bridges & $\Nreal$ & 250 \\
actually simulated cross-bridges & $\N$ & 250 \\
mass (macroscopic system) & $\Mr$ & 20 \\
mass (single cross-bridge) & $\Mq$ & $\frac{10}{\Nreal} = 0.04$ \\
stiffness (macroscopic system) & $\kappar$ & 1 \\
stiffness (single cross-bridge) & $\kappaq$ & $\frac{1}{\Nreal} = 0.004$ \\
initial position (macroscopic system) & $\ri$ & 1\\ 
initial velocity (macroscopic system) & $\si$ & 0\\ 
initial extensions (cross-bridges) & $\qd_j$ & $\sim \mathcal N(\mu=2, \sigma^2=1)$\\ 
initial distribution (cross-bridges) & $\pdensin(\qc)$ & $= \frac{1}{\sqrt{2 \pi}}\mathrm{exp}(-\frac{(\qc - 2)^2}{2} )$\\ 
time interval & $ $ & $[0,60]$.\\
\bottomrule
\end{tabular}
\caption{All simulations in the article are for these parameters. The parameters are chosen to demonstrate the mathematical structure, not to represent a realistic biological setting.}
\label{tab:params}
\end{table}

For systems with linear constraints, the use of the upwind method might appear exaggerated. Instead, it would be sufficient to approximate the
shift between $\pmst$ and the initial measure $\pmsin$, which is given by
$
w = \cDphi (r(t) - \ri).
$
Since the transport equation \cref{eq:muscle_transport_eq_first} has no source terms,
such a numerical scheme is fast and stable. In the presence of source terms, as in \cref{eq:muscle_transport_with_sources}, the resulting numerical scheme contains stiff differential equations. Since source terms are essential for realistic muscle models, we neglect this specialised method to focus on the upwind discretisation instead.

\subsection{Loss of Numerical Energy Conservation in the Partially Kinetic Description}

\index{Conservation!energy}
In \cref{fig:linear_energy}, we compare the energies of the ODE formulation
and the partially kinetic formulation with the same initial data as in \cref{subsec:mean_field_numerical}.
The results demonstrate an essential drawback of the numerical scheme for the kinetic description.
There are numerical methods for DAE formulation \cref{eq:dm_newton_r,eq:dm_newton_qj,eq:dm_constr_ind1} and the ODE formulation \cref{eq:dm_newton_eff,eq:dm_constr_ind1} 
which conserve the energy asymptotically  \cite{hairerGeometricNumericalIntegration2006}. In this particular linear example, the numerical conservation of energy is not very difficult even for stiff integrators.
However, numerical diffusion in the upwind scheme leads to an increase in the total energy,
as shown in \cref{fig:linear_energy}.

To explain the energy increase, we recall that the potential energy of the cross-bridges is
\begin{align}
E^\mathrm{pot}_\mathrm{\qd_1,\dots,\qd_\N} &= \frac{\kappaq}{2} \sum_{i=1}^N \qd_i^2 
\intertext{
which becomes 
}
E^\mathrm{pot}_\mathrm{\pmst} &=  \Nreal \frac{\kappaq}{2}\int_\Rq \qc^2 \dq \pmst
\end{align}
in the kinetic setting. The upwind scheme increases the potential energy artificially since numerical diffusion leads to a more widespread cross-bridge distribution, as demonstrated in \cref{fig:linear_long_density}. 
We are not aware of a method which conserves the total energy numerically and at the same time works in the presence of source terms.

\begin{figure}[h]
	\centering
	\includegraphics[width=0.45\textwidth]{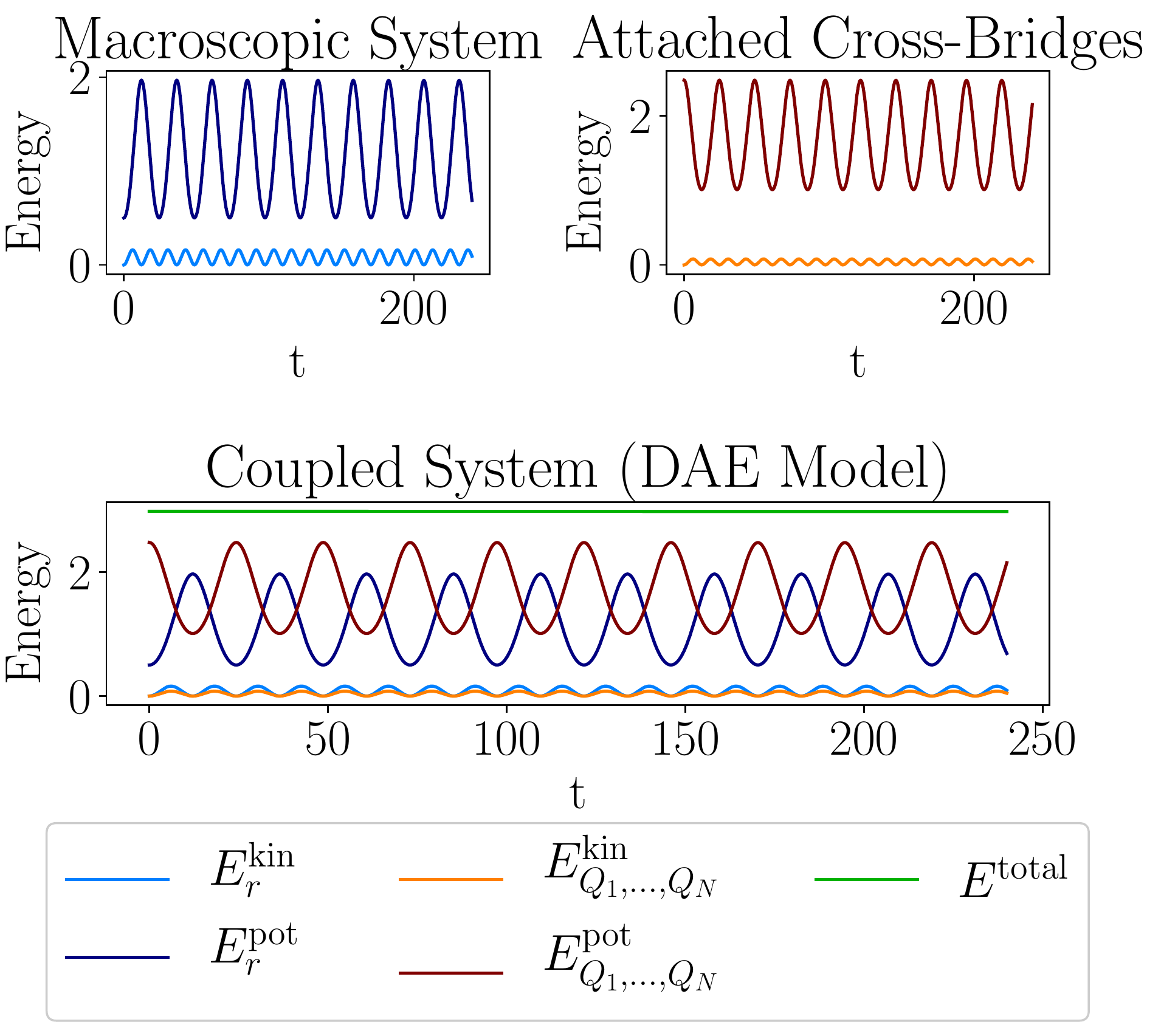}
	\hfill
	\includegraphics[width=0.45\textwidth]{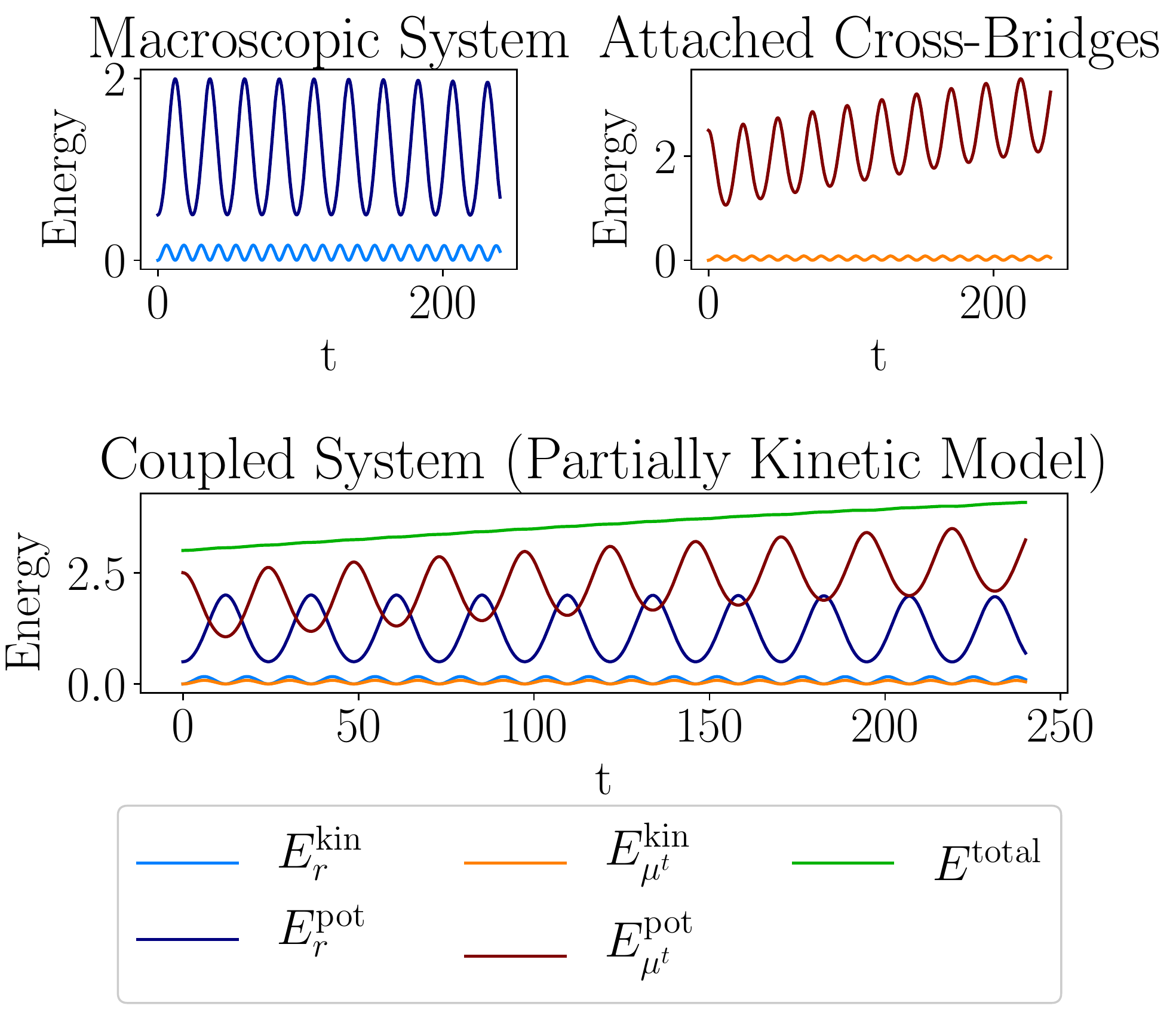}
	\caption{Energies of the discrete system (left) and the kinetic system (right). The kinetic energy is denoted by $T_r$ for the macroscopic system and $T_q$ for the particles, the potential energy is denoted by $U_r$ and $U_q$ respectively. The total energy $E_\mathrm{total}$ is well preserved in the discrete case, but not in the mesoscopic simulation.}
	\label{fig:linear_energy}
\end{figure}

\begin{figure}[h]
	\centering
	\includegraphics[width=0.6\textwidth]{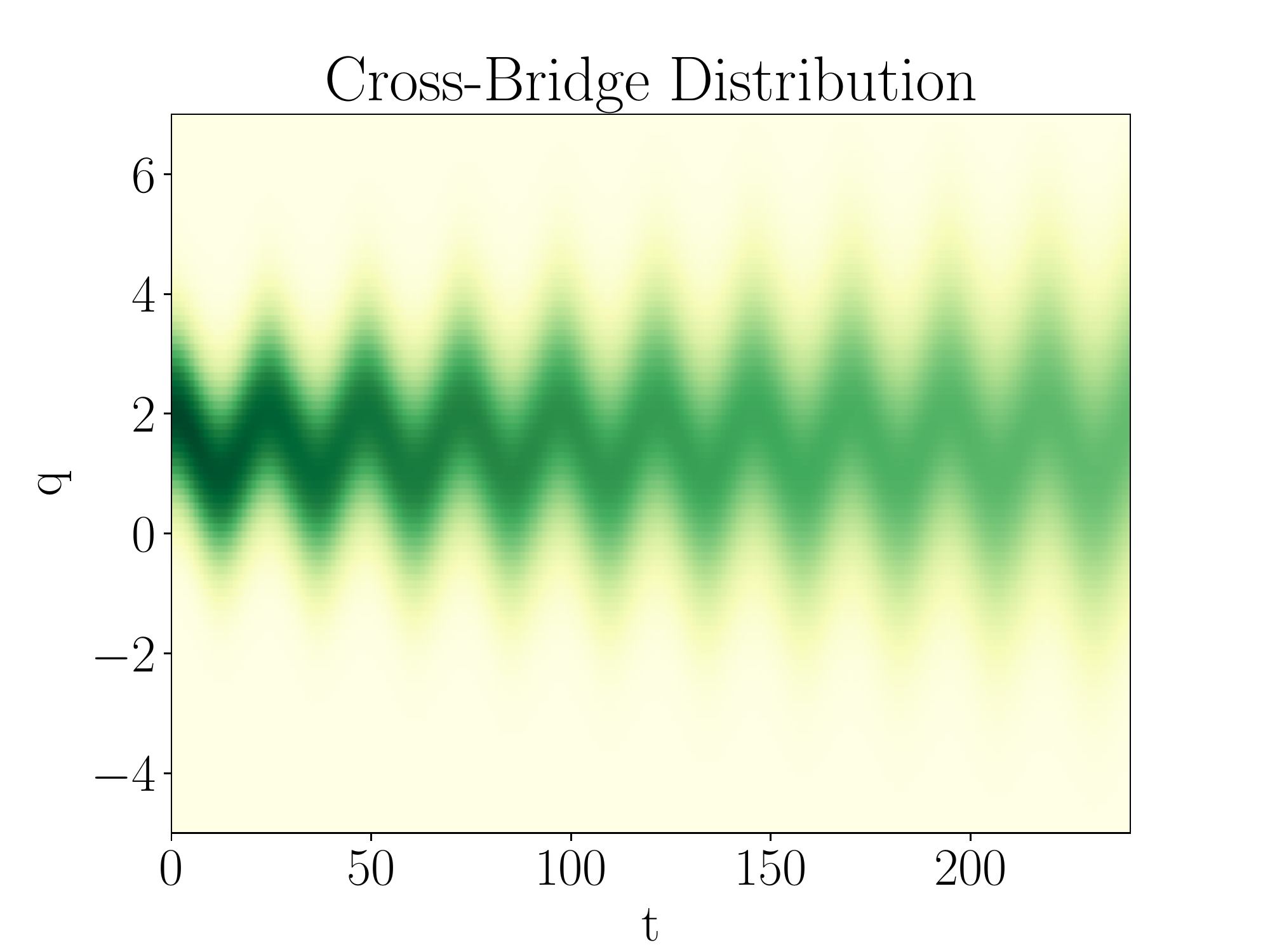}
	\caption{Numerical diffusion of the upwind scheme leads to a diffusion of the cross-bridge distribution. As a result, the potential energy of the cross-bridges increases.}
	\label{fig:linear_long_density}
\end{figure}

\section{Conclusion}

In this article, we have presented the new framework of linear partially kinetic systems. We have motivated this abstract class as a kinetic model for cross-bridge dynamics in skeletal muscles, in a manner which allows us to add constraints. The linear setting represents a toy example, for which ideas from kinetic theory are applicable. It can be argued that the analysis so far is restricted to a rather simple model scenario. 
Thus, there is a need to generalise these results to a wider class of models, which
then would yield a rigorous link between existing physiological models at different scales.

Finally, the numerics of partially kinetic systems is still in its infancy.
The investigation of the stability estimate was motivated by the lack of numerical analysis for linear partially kinetic systems.
We have presented an example in which the conservation of energy is violated, which already indicates the limitations of a naive discretisation.

\section*{Acknowledgement}

We thank Claudia Totzeck and Sara Merino-Aceituno for the fruitful discussions and hints regarding the mean-field limit. The work is motivated by the aim to develop a mathematical foundation for muscle tissue with a two-way coupling between cells and tissue. We thank Ulrich Randoll for his advice
and numerous discussion 
on the physiology of skeletal muscle tissue. This research is supported by the German Federal Ministry of Education and Research (BMBF) under grant no.  05M16UKD (project DYMARA).

\appendix

%
%

\bibliographystyle{spmpsci}      
\bibliography{ref}   

\printindex

\end{document}

%% file: definitions.tex
\newcommand{\RR}{\mathbb{R}}

\newcommand{\Op}[1]{\operatorname{#1}}

\newcommand{\Cal}[1]{\mathcal{#1}}

\newcommand{\Lip}{\mathrm{Lip}}

\let\dif\relax
\DeclareMathOperator*{\dif}{d}

\DeclarePairedDelimiterXPP\E[1]{\mathbb{E}}{[}{#1}{]}{}
\DeclarePairedDelimiterXPP\Var[1]{\Op{Var}}{[}{#1}{]}{}
\DeclarePairedDelimiterXPP\Prob[1]{\mathbb{P}}{[}{#1}{]}{}

\DeclarePairedDelimiterX\ScPr[2]{(}{)}{#1,#2}
\DeclarePairedDelimiterX\DuPr[2]{\langle}{\rangle}{#1,#2}

\newcommand{\qd}{{Q}}
\newcommand{\qc}{{q}}
\newcommand{\qg}{{q}}

\newcommand{\qcp}{{\qc}'}

\newcommand{\dq}[1]{\, \dif #1 (\qc)}
\newcommand{\dqp}[1]{\, \dif #1 (\qcp)}

\newcommand{\N}{N}
\newcommand{\Nreal}{N_\mathrm{real}}

\newcommand{\nr}{{n_{r}}}
\newcommand{\nq}{{n_{\qg}}}

\newcommand{\Rq}{{\mathbb{R}^{n_{\qg}}}}
\newcommand{\Rr}{\mathbb{R}^{n_{r}}}

\newcommand{\BRq}{{\mathcal{B}(\Rq)}}
\newcommand{\PRq}[1]{{\mathcal{P}^{#1}(\Rq)}}

\newcommand{\kappar}{\gamma_{{r}}}
\newcommand{\kappaq}{\gamma_{{q}}}

\newcommand{\cDphi}{G_r}

\newcommand{\pms}{\mu}
\newcommand{\pmst}{\pms^t}
\newcommand{\pmsin}{\pms^{\mathrm{in}}}

\newcommand{\pdens}{u}
\newcommand{\pdensin}{\pdens^{\mathrm{in}}}

\newcommand{\Mr}{M_r}
\newcommand{\Mq}{M_q}

\newcommand{\Fr}{{F_{\mathrm{r}}}}
\newcommand{\Fq}{{F_{\mathrm{q}}}}

\newcommand{\Meff}{{M_{\mathrm{eff}}}}
\newcommand{\Feff}{{F_{\mathrm{eff}}}}
\newcommand{\veff}{{v_{\mathrm{eff}}}}

\newcommand{\meff}{{m_{\mathrm{eff}}}}
\newcommand{\feff}{{f_{\mathrm{eff}}}}

\newcommand{\Meffn}[1][\N]{\Meff^{(#1)}}
\newcommand{\Feffn}[1][\N]{\Feff^{(#1)}}

\newcommand{\fmf}{{f_{\mathrm{mean}}}}

\newcommand{\empn}{\mu^{(emp)}}

\newcommand{\myhash} {\text{\#}}
\newcommand{\push}[2]{{#1}{\myhash}{#2}}

\newcommand{\ri}{r^{\mathrm{in}}}
\newcommand{\si}{s^{\mathrm{in}}}
\newcommand{\qi}{\qc^{\mathrm{in}}}
\newcommand{\qdi}{\qd^{\mathrm{in}}}
\newcommand{\xin}{x^{\mathrm{in}}}

\newcommand{\cflow}{Q}

\newcommand{\LipRq}{\mathrm{Lip}(\Rq)}